\documentclass[a4paper, 12pt]{amsart}
\usepackage[latin1]{inputenc}
\usepackage{amsfonts}
\usepackage{amssymb}
\usepackage[leqno]{amsmath}
\usepackage{amsthm}
\usepackage{amscd}
\usepackage[all]{xypic}
\usepackage[curve]{xy}
\usepackage{epsfig}
\usepackage{enumerate}
\usepackage{subfigure}

\newtheorem{theorem}{Theorem}[section]
\newtheorem{proposition}[theorem]{Proposition}
\newtheorem{definition}[theorem]{Definition}
\newtheorem{corollary}[theorem]{Corollary}
\newtheorem{problem}[theorem]{Problem}
\newtheorem{example}[theorem]{Example}
\newtheorem{lemma}[theorem]{Lemma}
\newtheorem{remark}[theorem]{Remark}

\newtheorem{conjecture}[theorem]{Conjecture}

\DeclareMathOperator{\Gal}{Gal}

\DeclareMathOperator{\Id}{Id}
\DeclareMathOperator{\Char}{char}

\def\dim{\mathrm{dim}}

\def\Ker{\mathrm{Ker\hspace{0.1cm}}}

\def\M{\mathcal{M}}
\def\op{\mathrm{op}}

\newcommand\al{\alpha}
\newcommand\be{\beta}

\newcommand\ra{\rightarrow}

\def\Z{\mathbb{Z}}
\def\R{\mathbb{R}}
\def\Q{\mathbb{Q}}

\def\ot{\otimes}

\def\bucle{\ar@(ur,ul)[]}
\def\point#1{*+[o][F-]{\text{\scriptsize $#1$}}}
\def\cuadrado#1{*+[F-]{\text{\scriptsize $#1$}}}
\def\dual{k[\xi]}

\newcommand{\abs}[1]{\left|#1\right|}
\newcommand{\dtext}[1]{\emph{\textbf{#1}}}

\title{Factorization structures with a 2-dimensional factor}
\author{\'{O}scar Cortadellas Izquierdo}
\address{Department of Algebra\\
University of Granada \\
Avda. Fuentenueva s/n\\ E-18071\\ Granada\\ Spain}
\email{ocortad@ugr.es}
\author{Javier L\'{o}pez Pe\~{n}a}
\address{Mathematics Research Centre\\ Queen Mary University of London\\ Mile End Road\\ London E1 4NS\\ United Kingdom}
\email{j.lopez@qmul.ac.uk}
\author{Gabriel Navarro}
\address{Department of Computer Sciences and AI\\
 University of Granada \\
C/ El Greco s/n\\ E-51002 \\ Ceuta\\ Spain}
\email{gnavarro@ugr.es}
\thanks{
Research supported by MTM2007-66666, FQM-1889 and FQM-266. J.~L\'opez was also supported by Max-Planck Institut f\"ur Mathematik in Bonn.}

\begin{document}

\begin{abstract}
We introduce the notion of quantum duplicates of an (associative, unital) algebra, motivated by the problem of constructing toy-models for quantizations of certain configuration spaces in quantum mechanics. The proposed (algebraic) model relies on the classification of factorization structures with a two-dimensional factor. In the present paper, main properties of this particular kind of structures are determined, and we present a complete description of quantum duplicates of finite set algebras. As an application, we obtain a classification (up to isomorphism) of all the algebras of dimension 4 (over an arbitrary field) that can be factorized as a product of two factors.
\end{abstract}

\maketitle

\section*{Introduction}

Consider a manifold $M$ representing some physical system. From a dual point of
view, this manifold can also be represented by some algebra of functions $A$
(that could be taken, for instance, to be $A=C^\infty(M)$, the algebra of smooth
functions on $M$) over some base field $k$ (usually $k=\mathbb{R}$ or
$k=\mathbb{C}$ when dealing with actual physical systems). Now, if we want to describe a physical system consisting on two
``parallel'' sheets $M_1$ and $M_2$, both equal to $M$, the natural algebra to
consider from a classical point of view is the direct product $A\times A$, which
is isomorphic to the algebra $A\ot k^2$.

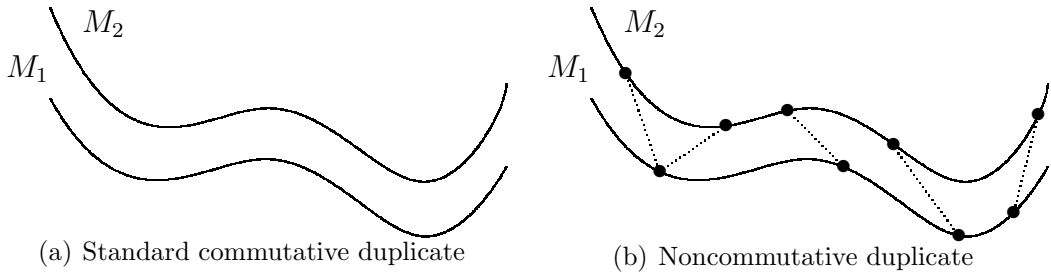
\begin{figure}[ht]
\!\!\!\subfigure[Standard commutative duplicate]
    {   $\displaystyle \xy (0,13) ;(60,4)
            **\crv{(10,-5)&(30,15)&(50,-15)&(60,4)};
            (0,25) ;(60,15)
            **\crv{(10,1)&(30,22)&(50,-8)&(60,12)};
            (-3,17)*+{M_1}; (7,23)*+{M_2}
            \endxy $
    }\
    \subfigure[Noncommutative duplicate]
    {   $\displaystyle
        \xy (0,13) ;(60,4)
        **\crv{(10,-5)&(30,15)&(50,-15)&(60,4)}
        ?(.65)*{\bullet}="A";
        ?(.1)*{\bullet}="B";
        ?(.4)*{\bullet}="C";
        ?(.8)*{\bullet}="D";
        (0,25) ;(60,15)
        **\crv{(10,1)&(30,22)&(50,-8)&(60,12)}
        ?(0.5)*{\bullet}="F";
        ?(0.05)*{\bullet}="G";
        ?(0.2)*{\bullet}="H";
        ?(0.3)*{\bullet}="I";
        ?(0.9)*{\bullet}="J";
        "F"; "A"  **\dir{.};
        "G"; "B"  **\dir{.};
        "H"; "B"  **\dir{.};
        "J"; "D"  **\dir{.};
        "I"; "C"  **\dir{.};
        (-3,17)*+{M_1};
        (7,23)*+{M_2}
        \endxy $
    }
\end{figure}

However, assume that we keep studying
this system while making the two sheets come closer and closer. After we reach
certain critical distance, one could expect that measurements on $M_1$ would
start to interfere with measurements on $M_2$, so that they do not commute
anymore. Under these premises, the algebra $A\ot k^2$ can no longer be assumed
(since it is commutative) to represent the measurements on the ``parallel, but
very close'' sheets $M_1$ and $M_2$. In the present work, our aim is to define
and study, inside an algebraic framework, a suitable replacement for the algebra $A\ot k^2$ in this quantized
situation, i.e. a (noncommutative) algebra that may be regarded as a
deformation of $A\ot k^2$ and retains similar structural properties.

In particular, we would like to keep the linear dimension of the algebra $A\ot
k^2$, since the fact that a system starts showing up quantum effects should not
alter the number of (linearly independent) quantities that we can meassure on
it. Thus, we are led to finding an algebra $X$, a ``\emph{quantum duplicate}''
of $A$, which is isomorphic, as a vector space, to $A\ot k^2$. In particular, if
$A$ is finite dimensional, the dimension of a quantum duplicate of $A$ should
be twice the dimension of $A$.

The method we propose in order to build these so-called quantum duplicates is
the use of a \dtext{factorization structure}, or \dtext{twisted tensor product} involving the algebra $A$ plus a
(necessarily commutative) 2--dimensional factor. The number of ways in which we
can choose this two dimensional factor depends on the field $k$. More
concretely, if $k$ admits a degree 2 field extension $\bar{k}$, (as happens, for
instance, with the real numbers), then we have three kinds of non-isomorphic
algebras of
dimension 2 (over $k$), namely:

\begin{enumerate}
    \item The trivial direct product $k^2$,
    \item Quadratic field extensions of $k$,
    \item The ring of dual numbers $\dual\cong k[x]/(x^2)$.
\end{enumerate}

On the other hand, if $k$ does not admit nontrivial extensions (for instance,
if $k=\mathbb{C}$), there are only two possible algebras: the direct product
$k^2$ and the dual numbers.

From a purely algebraic point of view, the notion of twisted tensor product comes directly from the \dtext{factorization problem}:

\begin{quotation} \it
Given some kind of (algebraic) object, is it possible to find to suitable subobjects, having minimal intersection and such that they generate our original object?
\end{quotation}

The factorization problem has been intensively studied in the case of groups, coalgebras and Hopf algebras, and algebras (cf. for instance \cite{Takeuchi81a}, \cite{Caenepeel00a}, \cite{Caenepeel02a}, \cite{Agore07a}). In the particular case of algebras, a well known result (independently proven many times) establishes a one-to-one correspondence between the set of factorization structures admitting two given algebras $A$ and $B$ as factors and the set of so-called \dtext{twisting maps}, which are linear maps $\tau:B\ot A\to A\ot B$ satisfying certain compatibility conditions with respect with the units and products of $A$ and $B$.

Henceforth, the problem of constructing factorization structures with given factors boils down to the problem of finding all the existing twisting maps for those factors. Under suitable, very mild, conditions (for instance, whenever $A$ and $B$ are affine algebras), the set $\mathcal{T}(A,B)$ of all the twisting maps $\tau:B\ot A \to A\ot B$ is an algebraic variety, and two interesting problems arise:

\begin{problem}
Is it possible to describe explicitly the variety $\mathcal{T}(A,B)$?
\end{problem}

\begin{problem}
Once the variety $\mathcal{T}(A,B)$ is known, is it possible to determine which points of the variety give rise to isomorphic algebras?
\end{problem}

These two problems, even in the simplest cases, turn out to be very difficult. Though there are many different methods that produce twisted tensor products of two given algebras, not a single one that produces all the existing ones is known, let alone describing the properties of the algebraic variety. Even harder is the problem of the determination of the isomorphism classes of algebras obtained from the same factors through different tensor products, or finding any description of these isomorphism classes in terms of the variety $\mathcal{T}(A,B)$.

Recently, Cibils showed in \cite{cibils} that the set $\mathcal{T}(k^{2},A)$ of twisted tensor products between any algebra $A$ and the commutative, semisimple algebra $k^{2}$ (also called the set of \dtext{2--interlacings}) is in one to one correspondence with couples of linear endomorphisms of the algebra $A$ satisfying certain conditions. If we take $A=k^n$, these couples of linear maps can be described by combinatorial means using certain families of colored quivers, and this description gives a simple way to describe all the twisted tensor products $k^n\ot_{\tau} k^{2}$, up to isomorphism (cf. \cite{cibils, Lopez08a}). Some other partial steps in the classification problem for factorization structures have been undertaken in \cite{Borowiec00a} and the final sections of \cite{Guccione99a}.

Beyond the physical motivation originating the notion of noncommutative duplicate (defined by Cibils), the remaining choices of two-dimensional factors have their own source of interest.
In particular, the building of duplicates by means of dual numbers has
yet another physical interpretation. Being the algebra of dual numbers one of the simplest examples of non-trivial superspaces, where the $\xi$--direction can be reinterpreted as the fermionic direction, and the scalar component as the bosonic direction, the procedure of  duplicating a manifold using dual numbers can be regarded as a simple way of adding a superstructure to the physical system described by the manifold. Moreover, quantizations of the tensor product $A\otimes k[\xi]$ admit an interpretation as \dtext{infinitesimal deformations} (for a \emph{central} formal parameter) in the framework of formal deformation theory (cf. \cite{Gerstenhaber64a}), so twisted tensor products $A\otimes_\tau k[\xi]$ may be regarded as infinitesimal deformations with respect to a non-central parameter, with the added advantage of the existence of such a kind of deformations for algebras that are rigid from the formal point of view (like separable algebras). Finally, the remaining case of quantum duplicates obtained using a quadratic field extension have similar properties to complexifications of real algebras, hinting the possibility of thinking about them as noncommutative scalar extensions.

In Section \ref{section1} we introduce the definition of \dtext{quantum duplicates} of an algebra $A$ as twisted tensor products $A\otimes_\tau B$ where $B$ is any two-dimensional algebra, characterize the
set of twisting maps as a set of couples $(f,\delta)$ where $f$ is an algebra endomorphism of $A$ and $\delta$ is an $f$-derivation, satisfying certain compatibility conditions. We also show how to lift certain classes of endomorphisms and involutions to such kind of twisted tensor products, with special attention to the case in which the algebra $B$ is a quadratic field extension of $k$, obtaining a simple criterion (similar to the one existing for telling whether a vector space is a complexification of a real one) for determining whether or not certain $k$-algebras factorize as quantum duplicates with a quadratic field extension.

Sections \ref{section2} and \ref{section3} deal with the classification problem of isomorphism classes of the resulting twisted tensor products. More concretely, in Section \ref{section2} we classify (up to isomorphism) all quantum duplicates of the finite set algebras $k^n$ by means of combinatorial techniques, obtaining results similar to the ones contained in \cite{cibils}.

The paper concludes with the complete classification, in Section \ref{section3}, of all the algebras of dimension 4 that can be obtained as a twisted tensor product, that should necessarily be of two factors of dimension 2. For the case of an algebraically closed field, the resulting algebras are displayed inside the diagram of all the four dimensional algebras, showing that no apparent pattern relating those \emph{factorizable} algebras appears.

Along this paper, $k$ will denote a field. All the algebras will be unital, associative $k$-algebras. The tensor product will be taken over $k$ and all maps will be $k$-linear. An algebra $X$ is a \dtext{factorization structure} of the algebras $A$ and $B$ if there exist two injective algebra maps $i_{A} : A \hookrightarrow X$ and $i_{B} : B \hookrightarrow X$ and the map $\varphi : A \ot B \rightarrow X$ defined by $\varphi(a \ot b) = i_{A}(a) \cdot i_{B}(b)$ is a linear isomorphism.

A $k$-linear map $\tau : B \ot A \rightarrow A \ot B$ is said to be a \dtext{(unital) twisting map} if
\begin{equation}\label{bd:1}
 \tau \circ (B \ot \mu_{A}) = (\mu_{A} \ot B) \circ (A \ot \tau) \circ (\tau \ot A)
\end{equation}
\begin{equation}\label{bd:2}
 \tau \circ (\mu_{B} \ot A) = (A \ot \mu_{B}) \circ (\tau \ot B) \circ (B \ot \tau),
\end{equation}
\begin{equation}\label{bd:3}
 \tau(1\ot a) = a \ot 1,\quad \tau(b\ot 1) = 1\ot b\quad \text{for all $a\in A$, $b\in B$}.
\end{equation}
where $\mu_{A}$ and $\mu_{B}$ stand for the product of $A$ and $B$ respectively and $A$ and $B$ stand for the identity maps on each algebra.

The map $\mu_{\tau} := (\mu_{A} \ot \mu_{B}) \circ (A \ot \tau \ot B)$ defines an associative product over $A \ot B$ if, and only if, $\tau$ is a twisting map. So we can endow the vectorial space $A \ot B$ with this product and get a new algebra $(A \ot B, \mu_{\tau})$, which will be denoted by $A\otimes_\tau B$.

\begin{proposition}[\cite{Cap95a}, \cite{Majid90b}]
 Let $(C, i_{A}, i_{B})$ a factorization structure of $C$ with factors $A$ and $B$. Then there exists a unique twisting map $\tau : B \ot A \rightarrow A \ot B$ such that $C$ is isomorphic to $A \ot_{\tau} B$ as a twisted tensor product.
\end{proposition}

When working with path algebras $kQ$ of quivers
$Q$ (cf. for instance \cite{Simson}) we shall assume that arrows are multiplied as if they were
maps, for instance, in the quiver
$$\xymatrix{\point{1} \ar[r]^-{\alpha} & \point{2} \ar[r]^-{\beta} & \point{3}}$$
the length-two path from $\xymatrix{\point{1}}$ to
$\xymatrix{\point{3}}$ will be written as $\beta\alpha$. All along
the work, the ideal $(Q_{\geq 2})$ of $KQ$ generated by
all paths in $Q$ of length greater than one will play a fundamental r\^{o}le. The quotient $kQ/(Q_{\geq 2})$ will be denoted by $kQ_{< 2}$. We
shall denote by $Q^{\op}$ the opposite quiver of $Q$, that is, the
quiver that has the same set of vertices while arrows are reversed.

\section{Generalities about quantum duplicates}\label{section1}

Let $A$ and $B$ be two
(unitary) $k$-algebras, with $\dim_k B=2$, so that we may consider it given as a quotient $B=k[x]/(p(x))$, where $p(x)$ is a
polynomial of degree two.
All along this work we write it as $p(x)=x^2-\alpha x+\beta$ where
$\alpha, \beta\in k$. We also denote by $q$ the polynomial $q(x)=x^2+\alpha x+\beta$, that is, the polynomial whose roots (in the algebraic closure of $k$) are the opposite to the ones of $p$.

\subsection{Basic definitions and properties}
Our purpose is to describe the twisting maps between $A$ and $B$,
that is, the $k$-linear maps
\begin{equation}\label{eq:twist} \tau: k[x]/(p(x))\otimes A\longrightarrow A\otimes
k[x]/(p(x)) \end{equation} verifying the twisting conditions (\ref{bd:1}) and (\ref{bd:2}).
Following the method developed in \cite{cibils}, it is worth noting
that $A\otimes k[x]/(p(x))\cong A[x]/(p(x))$ and then a twisting map
as in (\ref{eq:twist}) is determined by the values $\tau(x\otimes
a)$ corresponding to the product $xa$ in $A[x]/(p(x))$. For any
$a\in A$, we put $\tau(x\otimes a) = xa = \delta(a) + f(a)x$. Then,
	\begin{gather}
		x^2 a=(\alpha x-\beta)a=\alpha(\delta(a)+f(a)x)-\beta a
=\alpha \delta(a)-\beta a + \alpha f(a) x \\
		x(xa)=x(\delta(a) + f(a)x) = \delta^2(a) - \beta f^2(a) +
f\delta(a)x + \delta f(a) x + \alpha f^2(a)  x
	\end{gather}
for any $a\in A$. Thus the associativity $x^2
a=x(xa)$ produces the equalities
	\begin{gather}
		p(\delta)=\delta^2-\alpha \delta+\beta id_A=\beta f^2 \label{eq:1} \\
		f\delta+\delta f=\alpha (f-f^2) \label{eq:2}
	\end{gather}
As in the proof of \cite[Proposition 2.10]{cibils}, the condition
$x(ab)=(xa)b$ for any $a,b\in A$, produces that $f:A\rightarrow A$
is a morphism of algebras and $\delta:A\rightarrow A$ is a left $f$-derivation. It is also clear that if we have such an $f$
and $\delta$ verifying (\ref{eq:1}) and (\ref{eq:2}), the linear map
defined by $\tau(x\otimes a)=\delta(a)\otimes 1+f(a)\otimes x$ is a
twisting map between $A$ and $B$. In other words, we have proven the
following result.

\begin{lemma}\label{lemma:map_derivation} The set of twisting maps
$\tau: k[x]/(p(x))\otimes A\longrightarrow A
\otimes k[x]/(p(x))$ is in one-to-one correspondence with the set
of pairs $(f,\delta)$, where $f:A\rightarrow A$ is an endomorphism
of algebras and $\delta:A\rightarrow A$ is a left $f$-derivation, verifying the conditions
$$\text{$p(\delta)=\beta f^2$ and $f\delta+\delta f=\alpha
(f-f^2)$}.$$
\end{lemma}

A twisting tensor product of the form $A\otimes_\tau B$, where $B$ is an algebra of dimension two, will be referred by a \dtext{quantum duplicate} of $A$.

\begin{example} If $\alpha=1$ and $\beta=0$, then $B\cong k^2$ and
we recover the approach given by Cibils \cite{cibils}, since, in
that case, the twisting maps are in one-to-one correspondence with
the pairs $(f,\delta)$ verifying that $p(\delta)=\delta^2-\delta=0$
and $\delta f+f \delta=f-f^2$.
\end{example}

\begin{remark}\label{rk:charnot2}
The previous lemma admits a refinement when the characteristic of
$k$ is different from two. Taking the linear transformation
$\phi(x)=(\alpha/2) x+1$, we obtain that $p(\phi (x))=x^2+\gamma$
and $k[x]/(p(x))\cong k[x]/(x^2+\gamma)$. Thus equations \eqref{eq:1}, \eqref{eq:2} are rewritten as follows:
	\begin{gather}
		\delta^2=\gamma (f^2-id_A) \\
		f\delta+\delta f=0
	\end{gather}
\end{remark}

\vspace{1cm}

%
	
\subsection{Characterization of certain quantum duplicates}
	
	When we have a real vector space $V$, we can construct the \dtext{complexification} of $V$, called $V^\mathbb{C}$, as the tensor product $V\otimes_{\mathbb{R}} \mathbb{C}$	. The original vector space $V$ remains a \emph{real} vector subspace of $V^\mathbb{C}$, and can be recovered if we take advantage of the canonical conjugation map $\chi:V^\mathbb{C}\to V^\mathbb{C}$ given by $\chi(v\otimes z):=v\otimes \bar{z}$. When the two dimensional algebra $l=k[x]/(p(x))=k(\eta)$ is a quadratic (Galois) field extension, twisted tensor products $A\otimes_{\tau} l$ can be regarded as "noncommutative scalar extensions" to $l$. Our goal in this subsection is to stablish a result that guarantees us that a given $k$--algebra $B$ can be factorized as $A\ot_\tau l$ for some twisted tensor product $\tau$. In the forecoming, $l$ will be assumed to be some fixed (Galois) quadratic field extension of $k$, with $k$ a field of characteristic different from 2.
	
	\begin{lemma}
		Let $B$ a $k$--algebra endowed with a right $l$--module structure $B\ot l\to B$. Then the map $i:l\to B$ given by $z\mapsto 1_B\cdot z$ is an injective algebra map.
	\end{lemma}
	\begin{proof}
    It is a straightforward calculation.
	\end{proof}
	
	Assume now that $B$ is endowed with a conjugation map $\sigma:B\to B$ satisfying the following conditions:
	\begin{enumerate}
		\item $\sigma^2=\Id_B$,
		\item $\sigma(ab)=\sigma(a)\sigma(b)$ for all $a,b\in B$.
		\item $\sigma(a\lambda)=\sigma(a)\bar{\lambda}$ for all $\lambda\in l, a\in B$
	\end{enumerate}
	and let
	\[
		A:=B^\sigma=\{a\in B|\ \sigma(a)=a\}.
	\]
	We have an obvious algebra map (the inclusion map) $i_A:A\to B$.
	
	\begin{lemma}
		The mapping
			\begin{eqnarray*}
				\varphi:A\otimes l & \longrightarrow & B \\
				a\otimes z & \longmapsto & a\cdot z
			\end{eqnarray*}
		is a linear isomorphism.
	\end{lemma}
	
	\begin{proof}

        Let $b\in B$, consider the elements
        $$a_1:=\frac{b+\sigma(b)}{2} \qquad  \text{and} \qquad a_2:=\frac{b-\sigma(b)}{2\eta}$$
        then, obviously, $b=a_1+a_2\cdot \eta$. Now,
        $$\sigma(a_1)=\frac{\sigma(b)+\sigma(\sigma(b))}{2}=\frac{\sigma(b)+b}{2}=a_1
        \qquad \text{and} \qquad \sigma(a_2)=\frac{\sigma(b)-b}{-2\eta}=a_2$$
        and therefore $a_1,a_2\in A$. Thus $b=\varphi(a_1\otimes 1+a_2\otimes \eta)$ and $\varphi$ is surjective.

        Let $\alpha\in B\otimes C$, we may write $\alpha=a\otimes +b\otimes \eta$ with $a,b\in A$, and then $\varphi(\alpha)=\varphi (a\otimes 1)+\varphi(b\otimes \eta)=a+b\cdot i$. Assume $\varphi(\alpha)=0$, that is, $a+b\cdot \eta=0$. Applying $\sigma$, $a-b\cdot \eta=0$. As a consequence, $a=b=0$ and thus $\varphi$ is injective.

	\end{proof}
	
	As a consequence of the previous lemmas, by using \cite[Proposition 2.7]{Cap95a} (see also \cite{tambara90a} or \cite{majid95a}), we obtain the following corollary:
	
	\begin{corollary}
		An algebra $B$ factorizes as a twisted tensor product $A\otimes_{\tau} l$ for some twisting map $\tau$ if, and only if, $B$ is endowed with a right $l$-module structure and a conjugation map satisfying (1), (2) and (3).
	\end{corollary}
	
	\subsection{Lifting of endomorphisms and involutions}
	
	Let $A$ be an algebra, $\varphi:A\to A$ an algebra map, and $A\otimes_\tau B$ a quantum duplicate of $A$, with $B=k[x]/(p(x))$, induced by the couple $(f,\delta)$. The map $\varphi$ admits a natural lifting $\tilde{\varphi}:A\otimes_\tau B\to A\ot_\tau B$ defined by $\tilde{\varphi}(a\otimes b):=\varphi(a)\otimes b$.
	
	Assume that the lifting $\tilde{\varphi}$ is an algebra map. Then on the one hand we have
	\[
		\tilde{\varphi}(xa) = \tilde{\varphi}(f(a)x+\delta(a)) = \varphi(f(a))x+\varphi(\delta(a)),
	\]
	while, in the other one,
	\[
		\tilde{\varphi}(xa) = \tilde{\varphi}(x)\tilde{\varphi}(a) = x\varphi(a)=f(\varphi(a))x + \delta(\varphi(a)).
	\]
	Thus, we have stablished the following result:

	\begin{theorem}
		An endomorphism $\varphi$ of $A$ can be lifted to an endomorphism of the twisted tensor product $A\otimes_{(f,\delta)} B$ if, and only if
		\[
			f\varphi = \varphi f, \quad\text{and}\quad \delta\varphi=\varphi\delta.
		\]
		In other words, the set of endomorphisms of $A$ admitting a lifting to $A\otimes_{(f,\delta)} B$ coincides with the set of morphisms that simultaneously commute with $f$ and $\delta$.
	\end{theorem}
	

Assume now that $A$ is an algebra endowed with an involution $a\mapsto a^\ast$, and that $l=k(\eta)$ is a quadratic field extension of $k$, with $\Char(k)\neq 2$, so that the Galois automorphism $\eta\mapsto -\eta$ gives us an involutive $k$--automorphism of $l$. Assume also that we are given a twisting map $\tau$ induced from a couple $(f,\delta)$. Our purpose is to determine conditions in $f$ and $\delta$ in such a way that we can ensure the existence of an involution $j$ in $A\otimes_{(f,\delta)} l$.

First, recall that for any linear endomorphism $\varphi$ of $A$ we can define its conjugate $\overline{\varphi}$ by $\overline{\varphi}(a):=\varphi(a^\ast)^\ast$. An easy computation shows that if $f$ is an algebra map, then so is $\overline{f}$, and if $\delta$ is a left $f$-derivation, then the conjugate $\overline{\delta}$ is a right $\overline{f}$-derivation.

Any involution $j$ defined on a twisted tensor product $A\otimes_{\tau} B$ which is compatible with the ones existing in $A$ and $B$ must satisfy
	\begin{eqnarray*}
		j(a\otimes b) & = & j((a\otimes 1)(1\otimes b))= j(1\otimes b)j(a\otimes 1) = \\
		& = & (1\otimes j_B(b))(j_A(a)\otimes 1)=\tau(j_B(b)\otimes j_A(a)),
	\end{eqnarray*}
i.e., we necessarily have $j=\tau\circ (j_B\otimes j_A)\circ \tau_{BA}$, where $\tau_{BA}$ is the flip map. Henceforth, the involutions can be combined in a unique compatible way, provided that the twisting map $\tau$ satisfies the involutive condition $(\tau\circ (j_B\otimes j_A)\circ \tau_{BA})^2=\Id_{A\otimes B}$ (cf. \cite{VanDaele94a}). In our particular situation, where the involution in $B$ is given by $\eta\mapsto -\eta$ and the twisting map is giving in terms of $f$ and $\delta$, the lifting of the involution of $A$ becomes
	\begin{eqnarray*}
		(a\eta)^\ast & = & \overline{\eta}a^\ast = -\eta a^\ast = -f(a^\ast)\eta -\delta(a^\ast)
	\end{eqnarray*}
so the involutive condition is
	\begin{eqnarray*}
		a\eta & = & ((a\eta)^\ast)^\ast = (-f(a^\ast)\eta -\delta(a^\ast))^\ast = \\
		& = & \eta f(a^\ast)^\ast  + \delta(a^\ast)^\ast = \\
		& = & f(f(a^\ast)^\ast) \eta  + \delta(f(a^\ast)^\ast) - \delta(a^\ast)^\ast,
	\end{eqnarray*}
	for all $a\in A$. We have proved the following result:
	
\begin{proposition}
	If $(A,\ast)$ is an algebra with involution, $l=k(\eta)$ is a quadratic field extension (with Galois automorphism $\eta\mapsto -\eta$), and we have a twisting map induced by a couple $(f,\delta)$, then the involution of $A$ lifts in a compatible way to an involution of $A\otimes_{(f,\delta)} l$ if, and only if, the following conditions are satisfied:
	\begin{gather*}
		f\circ \overline{f} = \Id_A,\\
		\delta \circ \overline{f} = \overline{\delta}.
	\end{gather*}
\end{proposition}

\section{Quantum duplicates of $k^n$}\label{section2}

 In this section we
describe and classify all quantum duplicates of $k^{n}$ for
some natural number $n \geq 2$. Denote by $\{e_{1}, \ldots, e_{n}\}$
the canonical basis of $k^{n}$. Following Cibils' procedure \cite{cibils}, the set
of algebra morphisms $f:k^{n} \ra k^{n}$ is in one-to-one
correspondence with the set of  set maps $\varphi : \{e_{1}, \ldots,
e_{n}\} \ra \{e_{1}, \ldots, e_{n}\}$, where, to each set map
$\varphi$, we associate the algebra morphism $f$ defined as
	\begin{equation}\label{7}
		f(e_{i}) = \sum\limits_{\{e_{j}|\varphi(e_{j}) = e_{i}\}}e_{j}, \hspace{5mm}\mbox{for any }i = 1, \dotsc , n.
	\end{equation}
This is also in one-to-one correspondence with the set of quivers
with $n$ vertices verifying that from each vertex starts precisely
one arrow. From now on we denote by $Q_{f}$ the quiver associated to
the endomorphism $f$ (by meanings of $\varphi$). We recall that $Q_f$
has $\{e_{1}, \dotsc, e_{n}\}$ as set of vertices, and there is an
arrow $e_{i} \ra e_{j}$ in $Q_f$ if, and only if, $\varphi(e_{i}) =
e_{j}$. For the convenience of the reader we introduce the following notation:
	\begin{definition}
For any positive integer $s$, we say that a connected component of $Q_f$ is an $s$-cycle if the component has the following shape:
	
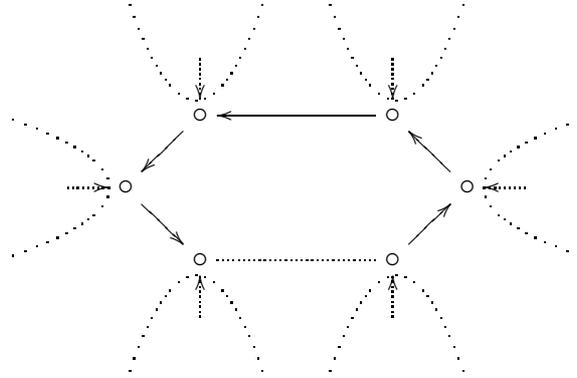
\[
		\xymatrix@C=15pt@R=15pt{
			& & \ar@{.}@/_41pt/[rr] & & & \ar@{.}@/_41pt/[rr] &          & &  & \\
			& & & \ar@{.>}[d] & & &\ar@{.>}[d] & & & \\
			\ar@{.}@/^41pt/[dd]& & & \circ \ar[ld] &   & &   \circ \ar[lll]& &  &  \ar@{.}@/_41pt/[dd] \\
			& \ar@{.>}[r]& \circ \ar[rd]   &  & &   &   & \circ \ar[lu] & \ar@{.>}[l]& \\
			& &    & \circ \ar@{.}[rrr] & & & \circ \ar[ru]  & & &\\
& & & \ar@{.>}[u]  & & &\ar@{.>}[u] & &  &\\
			& &   \ar@{.}@/^41pt/[rr] & & &\ar@{.}@/^41pt/[rr]
& & & &
 		}
	\]
where the central cycle has $s$ vertices and, by a  diagram
	
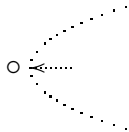
\[
		 \xymatrix@C=15pt@R=15pt{
		 	& &  \ar@{.}@/_41pt/[dd]  \\
			\circ &   \ar@{.>}[l] &      \\
              &   &
         }
    \]
we mean a tree with ascendent orientation. Observe that, for any $s$-cycle, if we remove the arrows of the central cycle, we obtain $s$ disjoint trees with ascendent orientation. For instance, in the figure below we have a 3-cycle.
$$
 
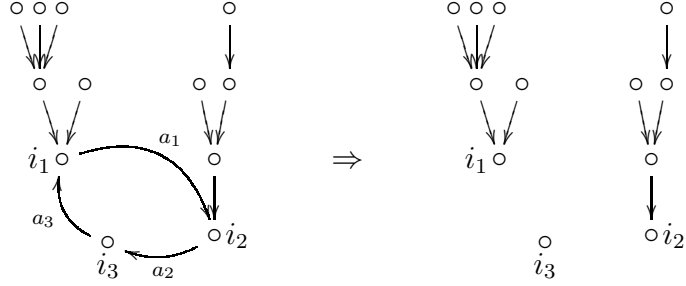
\xy
 (0,0)*+{\circ}="a1",(6,-11)*+{\circ}="a3",(20,-10)*+{\circ}="a2",
 (-3,0)*+{i_1},(6,-14)*+{i_3},(23,-10)*+{i_2},
 (-3,10)*+{\circ}="b",
 (3,10)*+{\circ}="c",
 (-6,20)*+{\circ}="d",
 (-3,20)*+{\circ}="e",
 (0,20)*+{\circ}="f",
 (20,0)*+{\circ}="h",
 (18,10)*+{\circ}="k",(22,10)*+{\circ}="l",(22,20)*+{\circ}="m",
 \ar@/^3.5ex/ "a1";"a2"^{a_1}\ar@/^1.3ex/"a2";"a3"^{a_2}\ar@/^1.7ex/
 "a3";"a1"^{a_3}
 \ar "b";"a1"\ar "c";"a1"\ar "d";"b"\ar "e";"b"\ar "f";"b"
 \ar "h";"a2"
 \ar "k";"h" \ar "l";"h" \ar "m";"l"
 \endxy
\qquad\Rightarrow\qquad
 \xy
 (0,0)*+{\circ}="a1",(6,-11)*+{\circ}="a3",(20,-10)*+{\circ}="a2",
 (-3,10)*+{\circ}="b",(-3,0)*+{i_1},(6,-14)*+{i_3},(23,-10)*+{i_2},
 (3,10)*+{\circ}="c",
 (-6,20)*+{\circ}="d",
 (-3,20)*+{\circ}="e",
 (0,20)*+{\circ}="f",
 (20,0)*+{\circ}="h",
 (18,10)*+{\circ}="k",(22,10)*+{\circ}="l",(22,20)*+{\circ}="m",
 \ar "b";"a1"\ar "c";"a1"\ar "d";"b"\ar "e";"b"\ar "f";"b"
 \ar "h";"a2"
 \ar "k";"h" \ar "l";"h" \ar "m";"l"
 \endxy
$$

We say that a connected component of $Q_f$ is an \dtext{strict $s$-cycle} if such component is
the quiver $\widetilde{\mathbb{A}}_s$, i.e., the quiver
$$
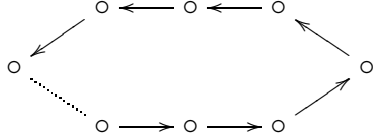
\xy \xymatrix@C=20pt@R=10pt{      & \circ \ar[ld] & \circ \ar[l] & \circ \ar[l]  &    \\
 \circ \ar@{.}[rd]  &  &    &   & \circ \ar[lu]    \\
  & \circ \ar[r] & \circ \ar[r] & \circ \ar[ru] &   }
\endxy$$
with $s$ vertices. Observe that, according to this nomenclature, an strict $1$-cycle component is nothing more than a single vertex with a loop whilst
  an strict $2$-cycle is the round-trip quiver.
\end{definition}

Therefore we have the following result, see \cite{jlns} or \cite{GTN}:

\begin{lemma} For any algebra map $f:k^n\rightarrow k^n$, each connected component of the quiver $Q_f$ is a $s$-cycle
for certain integer $s$.
\end{lemma}

With respect to the $f$-derivations, since $k^n$ is separable,
all of them are inner, so each $\delta \in \mbox{Der}(k^{n},\;
^{f}\!(k^{n}))$ is determined by certain element $a = (a_{1},
\ldots, a_{n})\in k^{n}$ such that
\begin{equation}\label{3}
\delta(e_{i}) =
(f(e_{i})-e_{i})a=\sum\limits_{\varphi(e_{j})=e_{i}}e_{j}a_{j} -
e_{i}a_{i}, \hspace{3mm}\mbox{for any }i=1, \ldots, n.
\end{equation}

Observe that if $e_i$ is a loop vertex, then $\delta(e_i)$
does not depend of the value of $a_i$. For that reason, we normalize
the element $a$ by taking $a_i=0$ for any loop vertex $e_i$. Then
each inner derivation is given by a unique normalized element of
$k^n$ (that, following the nomenclature of \cite{cibils}, we shall call a \dtext{coloration} of $Q_f$)

In order to characterize the derivations (=colorations) verifying
(\ref{eq:1}) and (\ref{eq:2}), we will need the following formulae:
\begin{gather} f^{2}(e_{i}) =
\sum\limits_{\varphi^{2}(e_{j}) = e_{i}}e_{j} \\
\delta^{2}(e_{i}) =
\sum\limits_{\varphi^{2}(e_{k})=e_{i}}a_{k}a_{\varphi(e_{k})}e_{k}
 - \sum\limits_{\varphi(e_{j})=e_{i}}a_{j}(a_{j} + a_{i}) e_{j} +
 a_{i}^{2}e_{i}\end{gather}
 for any vertex $e_i$ of $Q_f$. Therefore, we rewrite (\ref{eq:1}) and (\ref{eq:2}) as follows:

\begin{gather}\label{eq:3}\sum\limits_{\varphi^{2}(e_{k})=e_{i}}(a_{k}a_{\varphi(e_{k})}-\be)e_{k}
- \sum\limits_{\varphi(e_{j})=e_{i}}(a_{j}(a_{j} + a_{i})+\al
a_{j})e_{j} + (a_{i}^{2}+\al a_{i} + \be)e_{i}=0 \\
\label{eq:4} \sum\limits_{\varphi^{2}(e_{k})=e_{i}}
(a_k+a_{\varphi(e_k)}+\alpha) e_k-
\sum\limits_{\varphi(e_{j})=e_{i}} (a_j+a_i+\alpha)e_j=0
\end{gather}
for any vertex $e_i$ of $Q_f$. It is easy to see
 that the study of the possible colorations of the
  quiver $Q_f$ can be reduced to the study of each
   connected component separately. Let us start by showing up the case of an strict cycle.

\begin{proposition}\label{prop:1} Let $p(x)=x^2-\alpha x+\beta\in k[x]$ be a polynomial of degree two.
 The set of twisting maps $(f,\delta) : k^{n} \otimes k[x]/(p(x)) \ra k[x]/(p(x)) \otimes
k^{n}$, where $Q_f$ is a strict $s$-cycle for some $s>1$, is the
following:
	\begin{enumerate}
		\item If $Q_f$ is a strict 2-cycle, it
is a one-parameter family of twisting maps, indexed in the field,
 given by the colorations
 			\[
				\xymatrix{
					*+[o][F-]{\text{\scriptsize $a$}} \ar@/^4pt/@<0.5 ex>[r]& *+[o][F-]{\text{\scriptsize $b$}} \ar@/^4pt/@<0.5 ex>[l]
				}
				\hspace{0.5cm} \text{where} \hspace{0.5cm} a+b=-\alpha.
			\]
		\item If $Q_f$ is a strict $s$-cycle with $s>2$, each twisting map is given by coloring each vertex of the cycle with a root of $q$ satisfying that, if a vertex is colored by a root $r_1$, the immediate predecessor and the immediate successor of such vertex are colored by $r_2=-\alpha-r_1$.
			
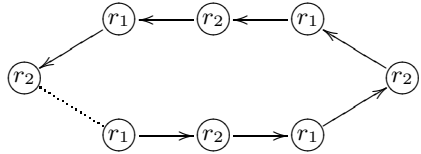
\[
				\xy
				\xymatrix@C=20pt@R=10pt{
					& \point{r_1} \ar[ld] & \point{r_2} \ar[l] & \point{r_1} \ar[l]  & \\
		 			\point{r_2} \ar@{.}[rd]  &  &    &   & \point{r_2} \ar[lu] \\
  					& \point{r_1} \ar[r] & \point{r_2} \ar[r] & \point{r_1} \ar[ru] &
				}
				\endxy
			\]
     As a consequence, if $s$ is even, there are as many colorations as different roots of $p$; and, if $s$ is odd, there exists a coloration (which would then be unique) if, and only if, $p$ has a unique (double) root. In particular, there are at most two twisting maps.
	\end{enumerate}
\end{proposition}

\begin{proof}
Let us suppose that we treat with a strict 2-cycle, i.e., a
connected component with shape
	\[
		\xymatrix{
			\circ \ar@/^4pt/@<0.5 ex>[r]& \circ \ar@/^4pt/@<0.5 ex>[l]
		}.
	\]
We also denote by $e_i$ and $e_j$ the vertices of this component. Then \eqref{eq:3} and \eqref{eq:4} applied to the vertex $e_i$ reduces to
	\begin{gather}
		\label{eq:5} (a_{i}a_{j} - \be)e_{i} - a_{j}(a_{j}+a_{i}+\alpha) e_{j} + (a_{i}^{2} + \al a_{i} + \be)e_{i} = 0 \\
		\label{eq:6} (a_i+a_j+\alpha) e_i- (a_i+a_j+\alpha) e_j=0
	\end{gather}
Clearly, these equations hold if and only if $a_i+a_j=-\alpha$.

Let us now consider an strict $s$-cycle with $s>2$. Then, for any
vertex $e_i$ of $Q_f$, the component $i$ of \eqref{eq:3} provide us
the equation $q(a_i)=0$. Therefore \eqref{eq:3} holds if, and only if,
each vertex is colored with a root of $q$. Now, denote by $e_j$ the
immediate predecessor of $e_i$. Then the component $i$ of
\eqref{eq:4} provides us the equation $a_i+a_j+\alpha=0$. Similarly,
we may prove that the immediate successor of $e_i$ must be colored
by $-\alpha-a_i$.
\end{proof}


\begin{theorem}\label{th:1}
	Let $p(x)=x^2-\alpha x+\beta\in k[x]$ be a polynomial of degree two. The set of twisting maps $\tau : k^{n} \otimes k[x]/(p(x)) \ra k[x]/(p(x)) \otimes
k^{n}$  is in one-to-one correspondence with the set of colored
quivers $Q_{f}$, where each connected component $Q_i$ of $Q_f$ is colored according to the following rules:
	\begin{enumerate}
		\item If $Q_i$ is a 1-cycle, any immediate predecessor of the loop vertex must be colored by a root of $q$. Given an immediate predecessor of the loop vertex $e$ colored by
$r_1$, other vertex $d$ in the same tree as $e$ must be colored by
$r_1$ if $t$ is even, and by $r_2=-\alpha-r_1$ if $t$ is odd, where
$t$ is the length of the path from $d$ to $e$.
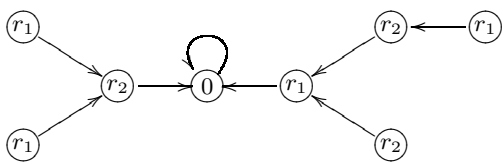
			\[
				\xymatrix@C=20pt@R=10pt{
					 \point{r_1}\ar[rd] &  &   &   &  \point{r_2} \ar[ld] &\point{r_1}\ar[l] \\
					 & \point{r_2} \ar[r]   & \point{0}\ar@(ur,ul)[]   &   \point{r_1} \ar[l] &   &  \\
					 \point{r_1}\ar[ru] &    &    &     &     \point{r_2} \ar[lu] &
				}
			\]
As a consequence, each connected component of this kind increases the number of twisting maps by at most a factor of $2^c$, where $c$ is the number of immediate predecessors of the loop vertex. Of course, if $q$ has no roots in $k$, the only connected components of this kind that may show up are strict 1-cycles.
		\item If $Q_i$ is a strict 2-cycle, any valid coloration must satisfy
			\[
				\xymatrix{*+[o][F-]
				{
					\text{\scriptsize $a$}} \ar@/^4pt/@<0.5 ex>[r] & *+[o][F-]{\text{\scriptsize $b$}} \ar@/^4pt/@<0.5 ex>[l]
				} \hspace{0.5cm} \text{where} \hspace{0.5cm} a+b=-\alpha.
			\]
			 Connected components of this kind give rise to a one-parameter family of twisting maps, the parameter ranging in the field,
		\item\label{bla} If $Q_i$ is either a non-strict 2-cycle or a $s$-cycle with $s> 2$, any valid coloration is determined by choosing a coloration of the central cycle. Each vertex of the central cycle must be colored by roots of $q$ in alternating order, meaning that, if a vertex is colored by a root $r_1$, any immediate predecessor and the immediate successor of such vertex are colored by $r_2=-\alpha-r_1$.
Moreover, for any vertex $e$ in the cycle, the tree attached to $e$ is also colored by alternating roots of $q$; i.e. if $e$ is colored by
$r_1$, any other vertex $d$ in the same tree as $e$ must be colored by
$r_1$ if $t$ is even, and by $r_2=-\alpha-r_1$ if $t$ is odd, where
$t$ is the length of the path from $d$ to $e$.
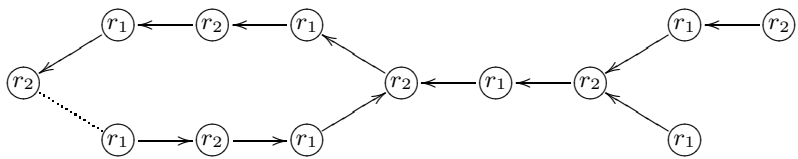
			\[
				\xy
				\xymatrix@C=20pt@R=10pt{
					& \point{r_1} \ar[ld] & \point{r_2} \ar[l] & \point{r_1} \ar[l] &  &  & &\point{r_1} \ar[ld] &  \point{r_2}\ar[l]\\
 					\point{r_2} \ar@{.}[rd]  &  &    &   & \point{r_2} \ar[lu]   & \point{r_1} \ar[l] & \point{r_2} \ar[l]& & \\
					& \point{r_1} \ar[r] & \point{r_2} \ar[r] & \point{r_1} \ar[ru] &  &  & & \point{r_1} \ar[lu]&
				}
				\endxy
			\]
     As a consequence, if $s$ is even, there are as many colorations
     as different roots of $p$; and, if $s$ is odd, there exists a (unique) coloration if, and only if, $p$ has a unique root.
In particular, this kind of components increase the number of twisting maps by at most a factor of $2$.
	\end{enumerate}
\end{theorem}

\begin{proof}
The second case is proven in Proposition \ref{prop:1}. Let us suppose
that $Q_f$ is a 1-cycle. The loop vertex is colored by 0 by the
arguments given above. Since any of the other vertices $e_i$ is not
inside a cycle, by applying the same arguments as in Proposition
\ref{prop:1}, we obtain that $q(a_i)=0$ and thus all the vertices
remaining are colored by a root of $q$. Now, let us consider an
arrow $\xymatrix{\point{e_i} \ar[r] & \point{e_j}}$ inside one of
the trees engaged to the loop vertex. Then (\ref{eq:4}) in the
component $i$ provides us the equation $a_i+a_j+\alpha=0$. That is,
for one of these trees, if one level is colored by a root $r_1$, the
next level is colored by $-\alpha-r_1$, see Figure \ref{arbolito}.
\begin{figure}[h]
        \[
 \xymatrix@C=15pt@R=15pt{\ar@{.}[rrr] & &                   & \point{r_1}  \ar@{.}[rrr]           &
         &     &       \\
 \ar@{.}[rr]    &        & \point{r_2} \ar[ur] \ar@{.}[rr]  &      & \point{r_2} \ar[ul]  \ar@{.}[rr] &
             & \\
 \ar@{.}[r] & \point{r_1} \ar[ru]    \ar@{.}[rr]    &                   & \point{r_1}
 \ar[lu]  \ar@{.}[rr]
          &            & \point{r_1} \ar[ul] \ar@{.}[r]  &   &  \text{where $r_2=-\alpha-r_1$}\\
 \ar@{.}[rrr] &           &
  & \point{r_2} \ar[u]      \ar@{.}[rr]        &            & \point{r_2} \ar[u]  \ar@{.}[r] & \\
 \ar@{.}[rr]  &          & \point{r_1} \ar[ur]  \ar@{.}[rr] &       & \point{r_1}
 \ar[lu]  \ar@{.}[rr]
   & &
             \\}
        \]
    \caption{Coloration of a tree engaged to the cycle\label{arbolito}}
    \end{figure}
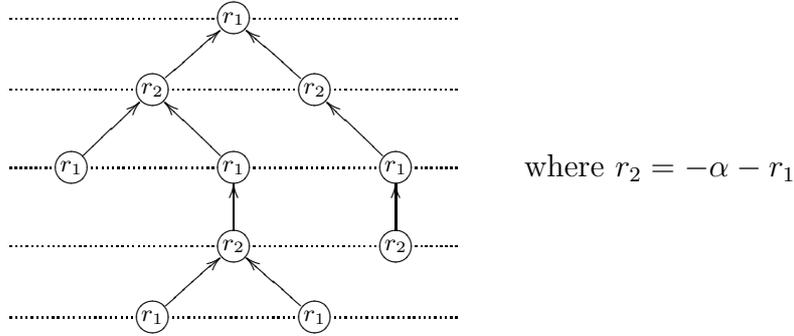

Finally, suppose that $Q_f$ is as stated in $(\ref{bla})$. Following
the same proof as in  Proposition \ref{prop:1}, we deduce that the
central cycle is colored by choosing to each vertex a root of $q$
and, if a vertex is colored by a root $r_1$, the immediate
     predecessor and the immediate successor of such
      vertex are colored by $-\alpha-r_1$. Now, applying the above
      reasoning to any of the trees engaged to a vertex of the
      cycle, we get the statement of $(\ref{bla})$. It is clear that
      these conditions are sufficient to have a coloration, thus the
      proof is completed.
\end{proof}

%
%
%
%

\begin{corollary}\label{cor:1} Let $p\in k[x]$ be an irreducible polynomial of degree two.
 The set of twisting maps $\tau : k^{n} \otimes k[x]/(p(x)) \ra k[x]/(p(x)) \otimes
k^{n}$  is in one-to-one correspondence with the set of colored
quivers $Q_{f}$ such that each connected component is either a
single loop vertex, or a round-trip quiver
$\xymatrix{*+[o][F-]{\text{\scriptsize $a$}} \ar@/^4pt/@<0.5 ex>[r]&
*+[o][F-]{\text{\scriptsize $b$}} \ar@/^4pt/@<0.5 ex>[l]}$
\hspace{0.05cm} where $a+b=-\alpha$.
\end{corollary}
\begin{proof}
By Theorem \ref{th:1}, if a connected component of $Q_f$ is neither
a strict 1-cycle nor a strict 2-cycle, the vertices are labeled by a
root of $q$ in $k$. Therefore there is no derivations for an algebra
map $f$ producing such components. Thus all components of $f$ must
be either a single loop vertex or a round-trip quiver. Now, apply
Theorem \ref{th:1} for such cases.
\end{proof}

%

\begin{corollary} Let $p\in k[x]$ be a polynomial of degree two such that it has a unique root $s$ in $k$.
 The set of twisting maps $\tau : k^{n} \otimes k[x]/(p(x)) \ra k[x]/(p(x)) \otimes
k^{n}$  is in one-to-one correspondence with the set of colored
quivers $Q_{f}$ such that:
\begin{enumerate}
	\item $Q_f$ is the quiver associated to a set map $\{1,\ldots
,n\}\rightarrow \{1,\ldots ,n\}$.
	\item The coloration of a connected component different of a
round-trip quiver is given by putting $0$ in the loop vertex and
$-s$ in the others.
	\item A round-trip connected component is colored by
 $\xymatrix{*+[o][F-]{\text{\scriptsize $a$}}
\ar@/^4pt/@<0.5 ex>[r]&
*+[o][F-]{\text{\scriptsize $b$}} \ar@/^4pt/@<0.5 ex>[l]}$,
\hspace{0.05cm} where $a+b=-2s$.
\end{enumerate}
\end{corollary}



Now we are going to describe the isomorphism classes of the algebras
that we have obtained. There is no loss of generality on assuming
that the quiver $Q_f$ is connected; otherwise we reason over every
connected component, and the resulting algebra will be the direct
product of the algebras associated to the distinct components (cf.
\cite{cibils}). As the reader should note, the case of a round-trip
quiver becomes very special. For that reason we study this case
separately.

\begin{proposition}\label{prop:2} Let $p(x)=x^2-\alpha x+\beta\in k[x]$ be a polynomial of degree two.
The twisted tensor product $k^{2} \otimes_{(f,\delta)}k[x]/(p(x))$,
where $(f,\delta)$ is the twisting map associated to the quiver
$$\xymatrix{\point{1} \ar@/^4pt/@<0.5
ex>[r]& \point{2} \ar@/^4pt/@<0.5 ex>[l]} \quad \text{colored by}
\quad \text{$\xymatrix{*+[o][F-]{\text{\scriptsize $a$}}
\ar@/^4pt/@<0.5 ex>[r]&
*+[o][F-]{\text{\scriptsize $b$}} \ar@/^4pt/@<0.5 ex>[l]}$,
\hspace{0.05cm} where $a+b=-\alpha$,}$$ is isomorphic to one of the
following algebras:
\begin{enumerate}
\item The matrix ring $\mathcal{M}_{2}(k)$ if $a$ and $b$ are not roots of $q$.
\item The quotient algebra $kQ_{< 2}$, otherwise, where $Q$ is the
round-trip quiver.
\end{enumerate}
\end{proposition}
\begin{proof}
Let us assume that $a$ and $b$ are not roots of $q$. In order to
show the isomorphism, it is enough to give the image of $e_1\otimes
1$, $e_2\otimes 1$ and $1\otimes x$ which, for brevity, we denote
simply by $e_1$, $e_2$ and $x$, respectively. Since $e_1+e_2=1$, we
map
$$
 e_1 \mapsto \left(\begin{array}{cr}
                    1 & 0 \\
                    0 & 0
                  \end{array}\right) \qquad
      e_2  \mapsto  \left(\begin{array}{cc}
                                         0 & 0 \\
                                         0 & 1
                                       \end{array}\right) \qquad
                                       x  \mapsto  \left(\begin{array}{cc}
                                         -a & 1 \\
                                         -q(a) & -b
                                   \end{array}\right)
$$
Since $q(a)\neq 0$, this provides us an isomorphism of algebras.

Now, suppose that $a$ and $b$ are roots of $q$. Let us denote the
elements of $Q$ as follows:
	\[
		\xymatrix@C=40pt{
			\point{u} \ar@/^4pt/@<0.5ex>[r]^-{R} & \point{v} \ar@/^4pt/@<0.5 ex>[l]^-{S}
		}
	\] 
Then we consider the isomorphism of $k$-algebras $\Phi: kQ_{<2}
\rightarrow k^{2} \otimes_{(f,\delta)}k[x]/(p(x))$ given by:
	\[ 
		u\mapsto e_1 \qquad v\mapsto e_2 \qquad R\mapsto e_2(x+b) = (x+a)e_1  \quad \text{and} \quad S\mapsto e_1(x+a) = (x+b)e_2
	\]
This is clearly well-defined. For instance, observe that
$$\Phi(R)\Phi(S)=(e_2x+be_2)(e_1x+ae_1)=-q(a)e_2=0,$$ as $a$ is a
root of $q$. This completes the proof.
\end{proof}

For the convenience of the reader, in order to state the main result of this section, we remind the following quiver-transformation introduced by Cibils
\cite{cibils}. Let us consider an 1-cycle $Q$ colored by means of the procedure explained above. Then we construct the forest $\widehat{Q}$ obtained as follows:
\begin{itemize}
\item Remove the loop vertex and all the arrows ending at it. This produces a finite number of trees.
\item Insert two isolated vertices, numbered by 1 and 2.
\item For each tree of the ones obtained above, insert an arrow from the root node to $1$ if that vertex is colored by $r_1$, or to $2$ if it is colored by $r_2$.
\end{itemize}
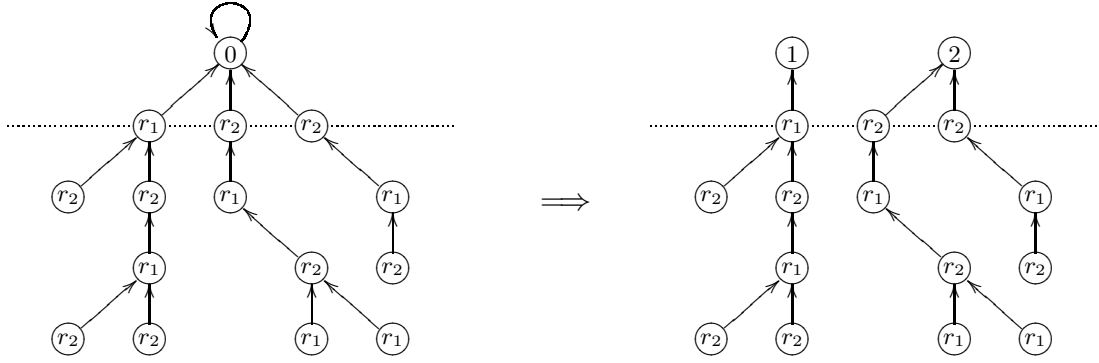
\begin{figure}[h]
        \[\xy
 \xymatrix@C=15pt@R=15pt{& &                   & \point{0} \ar@(ur,ul)[]         &
         &     &       \\
 \ar@{.}[rr]    &        & \point{r_1} \ar[ur] \ar@{.}[r]  &  \point{r_2} \ar[u]   \ar@{.}[r] & \point{r_2} \ar[ul]  \ar@{.}[rr] &
             & \\
& \point{r_2} \ar[ru]     &   \point{r_2}  \ar[u]            &
 \point{r_1} \ar[u]
          &            & \point{r_1} \ar[ul]  &   &  \\
  &           & \point{r_1} \ar[u]
  &            &  \point{r_2} \ar[lu]          & \point{r_2} \ar[u]  & \\
   &  \point{r_2} \ar[ur]        & \point{r_2}
 \ar[u]   &       & \point{r_1} \ar[u]
   & \point{r_1} \ar[lu] &
             \\}
 \endxy
\xy  (10,-20)*+{\Longrightarrow} \endxy \quad
 \xy
 \xymatrix@C=15pt@R=15pt{  &   &     \point{1}              &          &
  \point{2}       &     &       \\
 \ar@{.}[rr]    &        & \point{r_1}  \ar@{.}[r] \ar[u]  &  \point{r_2}  \ar[ru]  \ar@{.}[r] & \point{r_2}  \ar[u] \ar@{.}[rr] &
             & \\
& \point{r_2} \ar[ru]     &   \point{r_2}  \ar[u]            &
 \point{r_1} \ar[u]
          &            & \point{r_1} \ar[ul]  &   &  \\
  &           & \point{r_1} \ar[u]
  &            &  \point{r_2} \ar[lu]          & \point{r_2} \ar[u]  & \\
   &  \point{r_2} \ar[ur]        & \point{r_2}
 \ar[u]   &       & \point{r_1} \ar[u]
   & \point{r_1} \ar[lu] &
             \\}
       \endxy \]
    \caption{Cibils' transformation}
    \end{figure}

\begin{theorem}\label{th:qduplicateskn} Let $p(x)=x^2-\alpha x+\beta\in k[x]$ be a polynomial of degree two.
The twisted tensor product $R=k^{n} \otimes_{(f,\delta)}k[x]/(p(x))$,
where $(f,\delta)$ is the twisting map associated to a connected colored quiver $Q$,  is isomorphic to
the following algebra:
\begin{enumerate}[$(a)$]
\item When $Q$ is an strict 2-cycle,
\begin{enumerate}[$(i)$]
\item If $Q$ is colored by roots of $q$, then $R\cong kQ_{< 2}$.
\item Otherwise $R\cong \mathcal{M}_2(k)$.
\end{enumerate}
\item When $Q$ is a 1-cycle,
\begin{enumerate}[$(i)$]
\item If $p$ has no root, $n=1$ and therefore $R\cong k[x]/(p(x))$.
\item If $p$ has a single root, then $R\cong k(Q^{op})_{< 2}$.
\item If $p$ has two different roots, then $R\cong k((\widehat{Q})^{op})_{< 2}$.
\end{enumerate}
\item Otherwise, $R\cong k(Q^{op})_{< 2}$.
\end{enumerate}
\end{theorem}
\begin{proof}

$(a)$ is Proposition \ref{prop:2}. $(b)-(i)$ is given by Corollary \ref{cor:1} since $Q$ must be a single loop vertex. Then the statement follows trivially.

Let us prove $(b)-(ii)$ and $(c)$. We remind that each vertex in $Q^{\op}$ is the target of a unique arrow.
For any vertex $e_i$ in $Q^{\op}$, the twisting map given by $(f, \delta)$, say $\tau$, verifies that
\begin{equation}\label{eq:7} \tau(x\otimes e_i)=\delta(e_i)\otimes 1+f(e_i)\otimes x=\displaystyle\sum_{e_j}\epsilon_j(e_j\otimes 1)-\epsilon_i(e_i\otimes 1)+\displaystyle\sum_{e_j}e_j\otimes x,\end{equation}
where the $e_j$'s are the target of the arrows starting at $e_i$ in $Q^{\op}$ and, $\epsilon_i$ and $\epsilon_j$ are the labels of $e_i$ and $e_j$, respectively. Then we define a  morphism $\Phi: kQ^{\op}\rightarrow k^n\otimes_{(f,\delta)}
k[x]/(p(x))$ as follows:
\begin{itemize}
\item For any vertex $e_i$ in $Q^{\op}$, we set
$\Phi(e_i)=e_i\otimes 1$.
\item For any arrow $\alpha_i$ in $Q^{\op}$, we set $\Phi(\alpha_i)=e_i\otimes x+\epsilon_i (e_i\otimes 1)$,
where $e_i$ is the target of $\alpha_i$ in $Q^{\op}$, and $\epsilon_i$ is its color.
\end{itemize}
It is a straightforward calculation to prove that $\Phi$ is a
surjective algebra map such that $(Q^{\op}_{\geq 2})=\Ker \Phi$, and it is left to the reader. For instance, by using (\ref{eq:7}),
$$\Phi (\alpha_i)\Phi(e_i)=(e_i\otimes x+ \epsilon_i(e_i\otimes 1))(e_i\otimes 1)=
-\epsilon_i(e_i\otimes 1)+\epsilon (e_i\otimes 1)=0$$ if $e_i$ is not a loop vertex, whilst
$$\Phi (e_i)\Phi(\alpha_i)=(e_i\otimes 1)(e_i\otimes x+ \epsilon_i(e_i\otimes 1))=
(e_i\otimes x+ \epsilon_i(e_i\otimes 1))=\Phi (\alpha_i).$$

In order to prove $(b)-(iii)$ we use a reformulation of the morphism $\Phi$. Then
we define $\widehat{\Phi}:K(\widehat{Q})^{op}\rightarrow k^n\otimes_{(f,\delta)}k[x]/(p(x))$ as follows:
\begin{itemize}
\item For any vertex $e_i\neq \xymatrix{\point{1}}, \xymatrix{\point{2}}$; we set $\widehat{\Phi}(e_i)=e_i\otimes 1$.
\item $\widehat{\Phi}(\xymatrix{\point{1}})=e_0\otimes\left (\frac{x-r_1}{r_2-r_1}\right )$  and $\widehat{\Phi}(\xymatrix{\point{2}})=e_0\otimes\left (\frac{x-r_2}{r_1-r_2}\right )$, where
$e_0$ is the loop vertex in $Q$.
\item For any arrow $\alpha_i$ that does not start neither at $\xymatrix{\point{1}}$ nor $\xymatrix{\point{2}}$, we set $\widehat{\Phi}(\alpha_i)=e_i\otimes x+\epsilon_i (e_i\otimes 1)$,
where $e_i$ is the target of $\alpha_i$ and $\epsilon_i$ is its color.
\item For any arrow $\alpha_k:\xymatrix{\point{1}}\rightarrow e_k$, we set $\widehat{\Phi}(\alpha_k)=\alpha \left (e_k \otimes \frac{x-r_1}{r_2-r_1}\right )$
\item For any arrow $\alpha_k:\xymatrix{\point{2}}\rightarrow e_k$, we set $\widehat{\Phi}(\alpha_k)=\alpha \left (e_k \otimes \frac{x-r_2}{r_1-r_2}\right )$
\end{itemize}
We leave to the reader the details of the proof, see \cite[Theorem 4.2]{cibils}
\end{proof}

\begin{corollary}\label{irreducible}
Let $p\in k[x]$ be an irreducible polynomial of degree two. Any twisted tensor product $$k^{n}
\otimes_{(f,\delta)} \overline{k}\cong \mathcal{M}_2(k)^t\times
(\overline{k})^r,$$ where $\overline{k}=k[x]/(p)$ and $t$ and $r$
are the number of strict 2-cycles and strict 1-cycles in $Q_f$,
respectively.
\end{corollary}

\begin{theorem}
    Any quantum duplicate of $k^n$, $k^n \otimes_{\tau} k[k]/(p(x))$, where $p\in k[x]$ is a reducible polynomial, is isomorphic to
    an algebra of type
    \[
        (\mathcal{M}_2(k))^t \times kQ_{< 2},
    \]
    where $t$ is a natural number, and $Q$ an appropriate quiver.
\end{theorem}

\begin{corollary}
There is a finite number of quantum duplicates of $k^n$ for any fixed $n>2$.
\end{corollary}

\begin{example}
    Let us describe all quantum duplicates of $k^3$,
    $R=k^3\otimes_{\tau} k[x]/(p(x))$. If $\tau$ is given by a pair $(f,\delta)$,
    the only possibilities for the quiver $Q_f$ are:

\vspace{0.5cm}

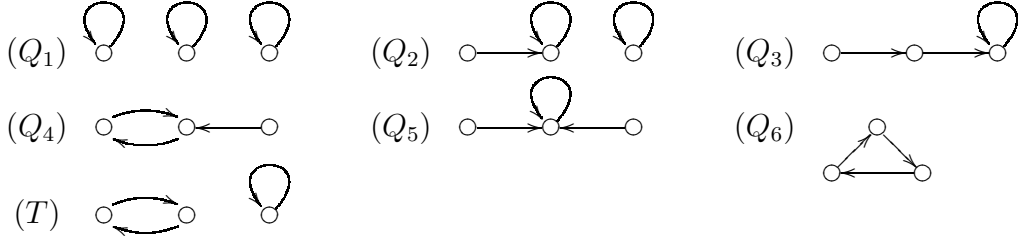
   \[
    \begin{array}{clclcl}
        (Q_1) &
            \xymatrix{
                *+[o][F-]{} \bucle &
                *+[o][F-]{} \bucle &
                *+[o][F-]{} \bucle
            } \qquad
        & (Q_2) &
            \xymatrix{
                *+[o][F-]{} \ar[r] &
                *+[o][F-]{} \bucle &
                *+[o][F-]{} \bucle
            } \qquad
        & (Q_3) &
            \xymatrix{
                *+[o][F-]{} \ar[r] &
                *+[o][F-]{} \ar[r] &
                *+[o][F-]{} \bucle
            } \vspace{0.5cm} \\
        (Q_4) &
            \xymatrix{
                *+[o][F-]{} \ar@/^4pt/@<0.5ex>[r] &
                *+[o][F-]{} \ar@/^4pt/@<0.5ex>[l] &
                *+[o][F-]{} \ar[l]
            }
        & (Q_5) &
            \xymatrix{
                *+[o][F-]{} \ar[r] &
                *+[o][F-]{} \bucle &
                *+[o][F-]{} \ar[l]
            }
        & (Q_6) &
            \xymatrix@R=10pt@C=10pt{
                & *+[o][F-]{} \ar[dr] \\
                *+[o][F-]{} \ar[ur] & & *+[o][F-]{} \ar[ll] \\
            } \vspace{0.2cm} \\
        (T)  &
            \xymatrix{
                *+[o][F-]{} \ar@/^4pt/@<0.5ex>[r] &
                *+[o][F-]{} \ar@/^4pt/@<0.5ex>[l] &
                *+[o][F-]{} \bucle
            }
       \end{array}
    \]

We should also consider the Cibil's transformation of the quivers with a 1-cycle component
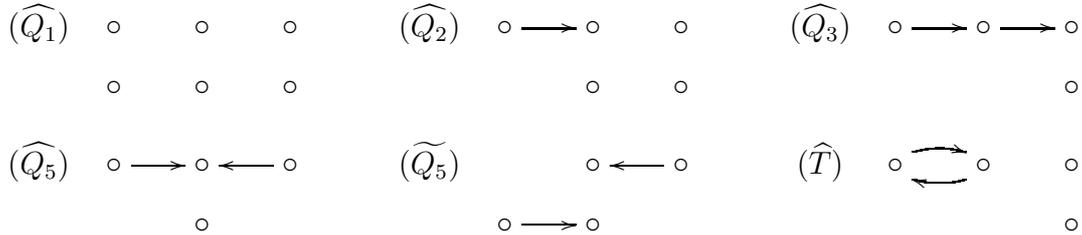
\[
    \begin{array}{clclcl}
        (\widehat{Q_1}) &
            \xymatrix@R=10pt@C=20pt{\circ & \circ & \circ \\
               \circ & \circ & \circ
            } \qquad
        & (\widehat{Q_2}) &
            \xymatrix@R=10pt@C=20pt{
                \circ \ar[r] & \circ & \circ \\
               & \circ & \circ
            } \qquad
        & (\widehat{Q_3}) &
            \xymatrix@R=10pt@C=20pt{
               \circ \ar[r] & \circ \ar[r] & \circ \\
                &  & \circ
            } \vspace{0.5cm} \\
        (\widehat{Q_5}) &
            \xymatrix@R=10pt@C=20pt{
                \circ \ar[r]& \circ &\ar[l] \circ \\
               & \circ &
            }
        & (\widetilde{Q_5}) &  \xymatrix@R=10pt@C=20pt{
             & \circ & \ar[l] \circ \\
               \circ \ar[r] & \circ &
            }
        & (\widehat{T}) &
            \xymatrix@R=10pt@C=20pt{
                \circ \ar@/^4pt/@<0.5ex>[r] &
                \circ \ar@/^4pt/@<0.5ex>[l] &
                \circ \\
                &  & \circ
            }
       \end{array}
    \]
where $\widehat{Q_5}$ and $\widetilde{Q_5}$ are the two possibilities depending on the coloration of $Q_5$.

By Corollary \ref{irreducible}, if $p$ is irreducible, $Q_f$ must be $Q_1$ or $T$. Then
$R\cong (k[x]/(p(x)))^3$ or $R\cong \mathcal{M}_2(k) \times k[x]/(p(x))$. Otherwise, if $Q_f\neq T$, $R$ is isomorphic to one of the (truncated) path
    algebras of the opposite quivers of $Q_1$, $Q_2$, $Q_3$, $Q_4$, $Q_5$ ,$Q_6$, $\widehat{Q_1}$, $\widehat{Q_2}$, $\widehat{Q_3}$, $\widehat{Q_5}$ or $\widetilde{Q_5}$. In case of $Q_f=T$, $R$ depends on the coloration of the round-trip component of $T$. Then we obtain that $R$ is the truncated path algebra of $T$ or $\widehat{T}$ if it is colored by roots of $q$; or $R\cong \mathcal{M}_2(k) \times k[x]/(x^2)$, or $R\cong \mathcal{M}_2(k) \times k^2$ if not.
    \end{example}

\section{Factorization structures of dimension 4}\label{section3}

The simplest nontrivial algebras that can be factorized as twisted tensor product ought to have factors of dimension at least 2, and thus the dimension of the product has to be greater or equal than 4. In the present Section, our purpose is to classify, up to isomorphism, all the algebras of dimension 4 that can be factorized. This problem turns out to have a strong dependence on the base field $k$. Basically we face two different situations:

\begin{enumerate}
	\item The field $k$ does not admit a field extension of degree 2 (for instance, if $k$ is algebraically closed). In this case, there are only two different algebras of dimension 2 that can appear as factors, the semisimple algebra $k^2$, and the algebra of dual numbers $k[\xi]$. In the realization of these algebras as quotients $k[x]/(p(x))$, the algebra $k^2$ corresponds to the cases in which $p$ has two distinct roots in $k$, and $k[\xi]$ to those for which $p$ has a double root.
	\item If $k$ admits quadratic extensions, or equivalently, there are polynomials of degree 2 with no roots in $k$, we have to take into account all these extensions as possible factors. The number of non-isomorphic quadratic extensions may range among one (for instance, if $k=\mathbb{R}$ the field of real numbers) to an infinite family of them (as happens for $k=\mathbb{Q}$ the field of rational numbers).
\end{enumerate}

Some of the possible combinations of the previous kinds of algebras to form 4 dimensional factorization structures have already been classified, concretely all the algebras of the form $k^2\otimes_\tau k^2$ (classified in \cite{cibils} and \cite{Lopez08a}), and all the factors of the form $k^2 \otimes_\tau A$, with $A=k[x]/(p(x))$ 2-dimensional, are described as a particular case of Theorem \ref{th:qduplicateskn}. For the sake of completeness, we relist the resulting algebras (without any proof).

\subsection{Twisted tensor products of the form $k^2\otimes_\tau k^2$}

Every twisted tensor product of the form $k^2\otimes_\tau k^2$ is isomorphic to one of the following algebras:

\begin{enumerate}
        \item The commutative algebra $k^4$.
        \item The algebra of matrices $\mathcal{M}_2(k)$.
        \item The quotient $kQ_{< 2}$ of the path algebra $kQ$ of
the round-trip quiver
    \[  Q=
        \xymatrix
        {   \circ \ar@/^4pt/@<0.5ex>[r]& \circ \ar@/^4pt/@<0.5 ex>[l]
        }
    \]
\item The path algebra $k\widetilde{Q}$ of the quiver
   \[  \widetilde{Q}=
        \begin{array}{l}
        \xymatrix@R=10pt@C=10pt
        {   & \circ &  \\
            \circ \ar[rr] & & \circ
        }
        \end{array}
    \]
\end{enumerate}

\subsection{Twisted tensor products of the form $k^2\otimes_\tau k[\xi]$}

Every twisted tensor product of the form $k^2\otimes_\tau k[\xi]$ is isomorphic to one of the following algebras:

\begin{enumerate}
	\item The commutative algebra $k[\xi]\times k[\xi]$.
	\item The algebra of matrices $\mathcal{M}_2(k)$.
	\item The quotient $kQ_{< 2}$ of the path algebra $kQ$ of
the round-trip quiver
    		\[  Q=
        		\xymatrix
        		{   	\circ \ar@/^4pt/@<0.5ex>[r]& \circ \ar@/^4pt/@<0.5 ex>[l]
        		}
    		\]
	\item The quotient $kQ'_{< 2}$ of the path algebra $kQ'$ of the quiver
		\vspace{1em}
		\[  Q'=
        		\xymatrix
        		{   \circ \ar[r]& \circ \bucle
        		}
		\]
\end{enumerate}

\subsection{Twisted tensor products of the form $k[\xi]\otimes_\tau k[\xi]$}

Since the algebra $k[\xi]$ of dual numbers is not separable, the description of twisting maps by means of endomorphisms and derivations is not as simple as it is when we use $k^2$. However, for this particular situation, the equations derived from the twisting conditions are simple enough to solve by a direct approach. If we consider the two copies of $k[\xi]$ respectively generated by $x$ and $y$ with $x^2=y^2=0$, and identify $x$ and $y$ with their images in the twisted tensor product $k[\xi]\otimes_\tau k[\xi]$, the twisting map is given by $yx = a + bx + cy + dxy$, and imposing the twisting conditions we obtain a system of equations that can be reduced to
	\begin{equation}
		\left\{
			\begin{array}{rcl}
				a (1+d) & = & 0 \\
				b & = & 0 \\
				c & = & 0
			\end{array}
		\right.
	\end{equation}
and thus the variety $\mathcal{T}(k[\xi],k[\xi])$ of twisting maps has two lines as irreducible components, giving rise to two one-parameter families of algebras, corresponding respectively to the solutions $a= 0$ and $d=-1$ of the previous system:
	\begin{gather}
		A_q:= k\langle x,y|\ x^2 = y^2 = 0,\; yx=qxy \rangle \\
		X_t:=k\langle x,y|\ x^2 = y^2 = 0,\; yx + xy = t \rangle	
	\end{gather}
both families intersect at the algebra $A_{-1} = X_0$ (which is the exterior algebra $\bigwedge k^2$ of $k^2$).

\begin{lemma}
	For any $t\neq 0$, the algebra $X_t$ is isomorphic to the matrix ring $\M_2(k)$.
\end{lemma}
\begin{proof}
	Just take the isomorphism given by
	\[
		x\mapsto \begin{pmatrix}
					0 & 1 \\
					0 & 0
				\end{pmatrix}, \quad
		y\mapsto \begin{pmatrix}
					0 & 0 \\
					t & 0
				\end{pmatrix}.
	\]
	
\end{proof}

The other irreducible component of the twisting variety actually gives rise to a complete family of non-isomorphic algebras. More concretely, we have the following result (whose proof is a straightforward computation):

\begin{lemma}
	The algebras $A_q$ and $A_h$ are isomorphic if, and only if, $q=h$ or $q=h^{-1}$.
\end{lemma}

\begin{remark}
	Remarkable algebras contained in the family $A_q$ are $A_{-1}=\bigwedge k^2$, the exterior algebra of $k^2$, and $A_1 \cong k[x,y]/(x^2,y^2)$, the only commutative algebra in the family, corresponding to the classical tensor product $k[\xi]\otimes k[\xi]$. It is also worth noticing that, unlike in the previous situations, the twisting variety $\mathcal{T}(k[\xi],k[\xi])$ is connected.
\end{remark}

\subsection{Twisted tensor products of the form $k^2\otimes_\tau l$}
Twisted tensor products of the form $k^2\otimes_\tau l$, with $l=k[x]/(p(x))$ a quadratic field extension of $k$, are again classified using Theorem \ref{th:qduplicateskn}, or more precisely its Corollary \ref{irreducible}. According to it, all the resulting algebras are of the following form:

	\begin{enumerate}
		\item The commutative algebra $k^2\otimes l\cong l\times l$.
		\item The matrix algebra $\M_2(k)$.
	\end{enumerate}

\subsection{Twisted tensor products of the form $k[\xi]\otimes_\tau l$}\label{ssec:duals_l}

For $l$ a quadratic field extension of $k$, by Lemma \ref{lemma:map_derivation}, twisted tensor products of the form $k[\xi]\otimes_\tau l$ are given by couples $(f,\delta)$, being $f$ an algebra endomorphism of $l$, and $\delta$ and $f$--derivation such that
	\begin{eqnarray}
		\delta^2 & = & 0 \label{ke_l:eq1}\\
		f\delta + \delta f & = & 0 \label{ke_l:eq2}
	\end{eqnarray}
Now, if $l$ is a Galois extension of $k$ (which is always the case if $\Char k \neq 2$), its $k$-linear endomorphisms are in one-to-one correspondence with elements of the Galois group $\Gal(l/k)\cong \Z_2$, which in this case correspond to the identity map and the nontrivial morphism $\sigma$ that exchanges the roots of $p(x)$. Moreover, since every Galois extension is separable, all derivations in $l$ are inner. If $f$ is the identity map on $l$, there are no nontrivial inner derivations, so we only get one couple $(\Id_l,0)$, that trivially satisfies \eqref{ke_l:eq1} and \eqref{ke_l:eq2}. The twisting map associated to the couple $(\Id_l,0)$ is the classical tensor product $k[\xi] \otimes l\cong l[\xi]$.

Let us describe the $\sigma$--derivations for the nontrivial morphism $\sigma$. Assume that $l$ is generated over $k$ by $\eta$, a root of the polynomial $p(x)=x^2-\alpha x + \beta$. As any derivation is inner, there exists $\theta\in l$ such that $\delta(x) = (\sigma(x) -x )\theta$, so in particular we get
	\[
		\delta(\eta)= (\sigma(\eta)-\eta)\theta= (\alpha - 2\eta)\theta.
	\]
For the study of this example it is useful to take into account the characteristic of the field. If $\Char k\neq 2$, as stated in Remark \ref{rk:charnot2}, the change of variables $x\mapsto (\alpha/2)x + 1$ takes the polynomial $p$ into $p'(x)=x^2 + \beta$, and $p'$ generates the same field extension $l$ as $p$, so we may just assume that $\alpha=0$ and obtain $\delta(\eta) = -2\eta\theta$. Writing $\theta=a+b\eta$, and taking into account that $\eta^2=-\beta$, we get
	\[
		\delta(\eta) = -2a\eta -2b\eta^2 = 2b\beta - 2a\eta.
	\]
Plugin this expression into equations \eqref{ke_l:eq1} and \eqref{ke_l:eq2} leads to conditions
	\begin{eqnarray}
		4a\eta & = & 0 \\
		4a^2\eta - 4 ab\beta & = & 0
	\end{eqnarray}
that are satisfied if, and only if, $a=0$. Thus, the only $\sigma$-derivations providing valid twisting maps are of the form $\delta_q(\eta) = q$ for some $q\in k$, and we get a 1-parameter family of twisting maps given by the couples $(\sigma,\delta_q)$, leading to the family of algebras
	\[
		B_q := k\langle x,y|\ x^2=0, y^2 = \gamma, xy+yx = q\rangle
	\]
where in order to simplify notation we are writing $\gamma$ in the place of $-\beta$.

\begin{lemma}\label{lm:matrices}
	For all $q\neq 0$, the algebra $B_q$ is isomorphic to the matrix algebra $\M_2(k)$.
\end{lemma}

\begin{proof}
	Just take the isomorphism given by
	\[
		x \mapsto 	\begin{pmatrix}
						0 & 0 \\
						q/\gamma & 0
					\end{pmatrix},\quad
		y \mapsto 	\begin{pmatrix}
						0 & \gamma \\
						1 & 0
					\end{pmatrix}.			
	\]
\end{proof}

\begin{lemma}\label{lemma:B0}
	The algebra $B_0$ is isomorphic to the invariant ring $\left(lQ_{< 2}\right)^G$, where $Q=
        		\xymatrix
        		{   	\circ \ar@/^4pt/@<0.5ex>[r]& \circ \ar@/^4pt/@<0.5 ex>[l]
        		}$ is the round-trip quiver, and $G$ denotes the group generated by the non-trivial automorphism that exchanges the vertices and the arrows of $Q$ and conjugate the scalar elements of $l$ with respect to the nontrivial element of the Galois group $\Gal(l/k)$.
\end{lemma}

\begin{proof}
Let us denote by $u, v$ the vertices of $Q$, and by $R, S$ its arrows.
First, realize that the $l$--algebra $B_0\otimes l=l\langle x,y|\ x^2=0, y^2 = \gamma, xy+yx = q\rangle$ is isomorphic to the algebra $lQ_{< 2}$ via the automorphism defined through
	\[
		x\longmapsto R + S,\quad y\longmapsto \sqrt{\gamma}u - \sqrt{\gamma}v.
	\]
	The automorphism $\sigma$ of the Galois group lifts in a trivial way to $B_0\otimes l$, and obviously $B_0\cong (B_0\otimes l)^\sigma$; now, it is straightforward to check that, under the above mapping, the image of $B_0$ is invariant under the action of $G$, resulting in the desired isomorphism $B_0\cong \left(lQ_{< 2}\right)^G$.
	
\end{proof}

Let us study now the case for $k$ a field with $\Char k= 2$. In this case, for the twisted derivation we get
	\[
		\delta(\eta) = \alpha a + \alpha b\eta,
	\]
which leads to equations
	\begin{eqnarray}
		\alpha^2 b & = & 0 \\
		\alpha^2 ab + \alpha^2 b^2\eta & = & 0,
	\end{eqnarray}
that are satisfied if, and only if, $b=0$ or $\alpha = 0$. Assume that $\alpha \neq 0$, that is, that the polynomial $p$ has two distinct roots on $l$; this is always the case if $k$ is a perfect field, otherwise the situation is very different, and will be treated separately. Then from the above equations we obtain $b=0$, yielding $\delta(\eta)=q \in k$, which gives us exactly the same family of algebras
	\[
		B_q := k\langle x,y|\ x^2=0, y^2 = \gamma, xy+yx = q\rangle
	\]
previously mentioned. Same proof as in Lemma \ref{lm:matrices} tells us that $B_q\cong \M_2(k)$ whenever $q\neq 0$, but in this case for $q=0$ the algebra $B_0$ is commutative and thus $B_0\cong k[\xi]\otimes l\cong l[\xi]$.

The remaining case, namely when $\Char k=2$ and $p(x)=x^2 + \beta$, is a special one. Realize that in this case the irreducible polynomial of $\eta$ over $k$ has $\eta$ as its unique (double) root. Under these premises, $l$ is not a Galois extension of $k$, since it fails to be separable, and the former reasoning cannot be applied. In particular, as there is only one root in $p$, there are no nontrivial automorphisms of $l$, so the only possible choice for $f$ is the identity; however, as the extension is not separable anymore, now we must take into account that there can be derivations that are not inner. In the present situation, with $f=\Id$, equations \eqref{ke_l:eq1} and \eqref{ke_l:eq2} become simply
	\begin{eqnarray}
		\delta^2 & = & 0, \\
		2\delta & = & 0,
	\end{eqnarray}
of which the second one is trivially satisfied, Thus, we only need to classify derivations of square 0.

For any derivation $\delta$ we have
	\begin{eqnarray}
		\delta(1) & = & \delta(1^2)=2\delta(1)= 0,\\
		\delta(\eta) & = & a \eta + b
	\end{eqnarray}
for some $a, b\in k$. It is straightforward to check that any choice of $a$ and $b$ leads to a valid derivation in $k(\eta)$, and we have that $\delta^2=0$ if, and only if,
	\[
		0=\delta^2(\eta)=\delta(a\eta + b)=a^2 \eta + ab,
	\]
and this equation is satisfied if, and only if, $a=0$. Henceforth, we have a one-parameter family of square 0 derivations that give us twisting maps, and these derivations are given by $\delta(\eta)=b$. Once again, the resulting twisted tensor products are the family of algebras
	\[
		B_q := k\langle x,y|\ x^2=0, y^2 = \gamma, xy+yx = q\rangle,
	\]
so we obtain the same twisted tensor products regardless of the characteristic of the field.

\subsection{Twisted tensor products of the form $l\otimes_\tau l'$}

In order to classify twisted tensor products of two field extensions, we shall initially assume that the characteristic of the base field is different from 2. As it happened in the last case, for fields of characteristic 2 weird phenomena may show up, so we will study them separately.

So, take $k$ such that $\Char k\neq 2$, and assume (without loss of generality) that the field extensions $l$ and $l'$ are given as splitting fields of the polynomials $x^2 - \alpha$ and $x^2 - \beta$, respectively, and let us denote by $\eta$, $\zeta$ their respective generators over $k$, so that $l=k\langle \eta|\ \eta^2=\alpha \rangle$, and $l'=k\langle \zeta|\ \zeta^2=\beta \rangle$. Using again Lemma \ref{lemma:map_derivation}, twisting maps $\tau: l'\ot l \to l\ot l'$ are in one-to-one correspondence with couples $(f,\delta)$ where $f$ is a $k$--linear endomorphism of $l$, satisfying the corresponding compatibility conditions. Since $l$ is a Galois extension of $k$ of degree $2$, there are only two $k$--linear endomorphisms of $l$, namely the identity and the map $\sigma$ given by $\sigma(\eta)=-\eta$. Since $l$ is separable, all derivations are inner, and thus there are no nontrivial derivations associated to the identity map. For the map $\sigma$, we need to find all the $\sigma$--derivations satisfying
	\begin{eqnarray}
		\delta^2 & = & 0 \\
		\delta\sigma + \sigma \delta & = & 0,
	\end{eqnarray}
where we are using the fact that $\sigma^2 = \Id_l$. Realize that the above conditions are trivially satisfied for the identity morphism and the trivial derivation (yielding the usual tensor product $l\ot l'$). If the derivation $\delta$, which is inner, is induced by an element $\theta=a+b\eta$, this leads us again to the same equations that showed up in the former paragraph:
	\begin{eqnarray}
		4a\eta & = & 0 \\
		4a^2\eta + 4 ab\alpha & = & 0
	\end{eqnarray}
with solutions given by $\delta(\eta)=q\in k$, and once again we obtain the 1--parameter family of algebras
leading to the family of algebras
	\[
		C_q := k\langle x,y|\ x^2=\alpha, y^2 = \beta, xy+yx = q\rangle.
	\]
However, the classification of these algebras is not as easy as in the former situations, and depends strongly on the ground field $k$. As a first approximation for the classification of the algebras $C_q$, we have the following result:
	\begin{lemma}
		The algebra $C_q$ is isomorphic to the generalized quaternion algebra $^{\alpha}k^{t}$, with $t=\frac{q^2-4\alpha\beta}{4\alpha^2}$. In particular, the algebras $C_q$ form a family of linked quaternion algebras.
	\end{lemma}
	\begin{proof}
		Take the isomorphism $C_q\to {^{\alpha}}k^{t}$ given by
		\[
			x \mapsto i,\quad y \mapsto \frac{q}{2\alpha}i + ij,
		\]
		where $i$ and $j$ are the generators of $^{\alpha}k^{t}=k\langle i,j|\ i^2 = \alpha, j^2 = t, ij + ji = 0\rangle$.
	\end{proof}
	We can draw some immediate consequences from the previous lemma. Firstly, it is a well known fact that the quaternion algebra $^{\alpha}k^{t}$ is a central simple algebra whenever $\alpha, t \neq 0$, and since in our present situation $\alpha\neq 0$, and $t=0$ if, and only if, $q^2 = 4\alpha \beta$; which might happen for at most two values of $q$. Henceforth, all the algebras $C_q$ (except maybe two of them) are central simple. In particular, since $\Char k\neq 2$, the field $k$ must have at least three elements, and we can assure that there always exists a twisted tensor product $l\ot_\tau l'$ which is simple, giving some supporting evidence to the following conjecture, due to J. Gómez-Torrecillas and F. van Oystaeyen:
	\begin{conjecture}
		For any algebra $A$, there exists a twisting map $\tau:A\ot A \to A \ot A$ such that $A\ot_\tau A$ is simple.
	\end{conjecture}
The solution to the isomorphism problem for quaternion algebras over a generic field is not explicitly known; however, we have the following result establishing necessary and sufficient conditions for two linked quaternion algebras to be isomorphic:
	\begin{lemma}
		Two linked quaternion algebras $^{a}k^{b}$ and $^{a}k^c$ are isomorphic if, and only if, $b/c\in N_{l/k}(l^{\times})$, being $l=k(\sqrt{a})$, and $N_{l/k}:l\to k$ the norm map of the extension $l/k$. As a consequence, $^{a}k^{b}$ is a matrix ring if, and only if, $b\in N_{l/k}(l^{\times})$.
	\end{lemma}
This result, as well as some others dealing with the problem of classifying quaternion algebras, can be found in \cite[Section 1.7]{Pierce82a} (cf. also \cite[Chapter III]{Lam05a} for a more recent revision). Applied to our concrete situation, and taking into account that for a field extension $l=k(\sqrt{\alpha})$ the norm map is given by $N_{l/k}(x+y\sqrt{\alpha})=x^2 - \alpha y^2$, we obtain the following result:
	\begin{theorem}\label{thm:quaternions} Let $q,h\in k$ such that $4\alpha\beta -q\neq 0$, $4\alpha\beta - h\neq 0$.
		\begin{enumerate}
			\item The algebras $C_q$ and $C_{h}$ are isomorphic if, and only if, there exist $x,y\in k$ such that
				\begin{equation}
					x^2 - \alpha y^2 = \frac{q^2 - 4\alpha\beta}{h^2 - 4\alpha\beta}
				\end{equation}
			\item $C_q$ is isomorphic to the matrix ring $\M_2(k)$ if, and only if, there exist $x,y\in k$ such that
				\begin{equation}
					x^2 - \alpha y^2 = q^2 - 4\alpha\beta
				\end{equation}
		\end{enumerate}
	\end{theorem}
In other words, the isomorphism classes of twisted tensor products of the form $l\otimes_\tau l'$ are given by:
	\begin{enumerate}
		\item The orbits of the action of $N_{l/k}(l^\times)$, seen as a multiplicative subgroup of $k^\times$, that intersect the image of the map $q\mapsto q^2 - 4\alpha\beta$.
		\item The algebra $C_q\cong C_{-q}$, provided that $q=2\sqrt{\alpha\beta}$ belongs to $k$.
		\item  The commutative algebra $l\otimes l'$.
	\end{enumerate}
	The degenerate case $(2)$ only shows up under very special conditions. Assume that $\sqrt{\alpha\beta}\in k$. Since obviously $\sqrt{\alpha}\in l=k(\sqrt{\alpha})$, this means that also $\sqrt{\beta}\in l$, and henceforth $l=l'$. Conversely, if $l=l'=k(\sqrt{\alpha})$, we may take $q=2\alpha$ and obtain the exceptional algebra.
	\begin{proposition}
		Let $l=l'=k(\sqrt{\alpha})$; the algebra
		\[
			\overline{C}_{2\alpha}:=l\ot C_{2\alpha}=l\langle x,y|\ x^2 = y^2 = \alpha, xy+yx=2\alpha \rangle.
		\]
	is isomorphic to the truncated path algebra $lQ_{< 2}$ of the round-trip quiver $ Q=
        		\xymatrix
        		{   	\circ \ar@/^4pt/@<0.5ex>[r]& \circ \ar@/^4pt/@<0.5 ex>[l] \label{eq:conic}
        		}$.
		
		As a consequence, we have an isomorphism between $C_{2\alpha}$ and the algebra of invariants $\left(lQ_{< 2}\right)^G$, being $G$ the group generated by the nontrivial automorphism of $kQ$ that conjugates scalars while exchanging vertices and arrows of $Q$.
	\end{proposition}
	\begin{proof}
		The proof of this result follows the same lines as the one of lemma \ref{lemma:B0}. In this case, we use the isomorphism $C_{2\alpha}\ot l \to lQ_{< 2}$ given by
		\[
			x\longmapsto \sqrt{\alpha}u - \sqrt{\alpha}v + R + S, \quad y\longmapsto \sqrt{\alpha}u - \sqrt{\alpha}v.
		\]
		
	\end{proof}
	
	For the family of central simple algebras $(1)$, as previously mentioned, the number of isomorphism classes (or orbits of the group action) depends strongly on the ground field $k$. A nice recent survey on quaternion algebras over different ground fields can be found in \cite{Lewis06a} (cf. also \cite{Pierce82a} and \cite{Lam05a}). Recall that any quaternion algebra must be isomorphic to either a division ring over $k$, or to the matrix ring $\M_2(k)$ (this can easily be proven by using Artin-Wedderburn structure theorem). For some familiar fields, a more concrete description can be given:
	\begin{enumerate}
		\item If $k$ is the field $\R$ of real numbers, $\alpha, \beta<0$ then the isomorphism class of $C_q\cong {}^\alpha\R^b$, with $b=(4\alpha\beta - q^2)/4\alpha$, depends on the sign of $b$, which is the sign of $q^2 - 4\alpha\beta$, More concretely, if $\abs{q}>2\sqrt{\alpha \beta}$, then $C_q\cong \M_2(\R)$. If, on the other hand, $\abs{q}<2\sqrt{\alpha \beta}$, then $C_q\cong\mathbb{H}$, the usual quaternion algebra.
		\item If $k=\mathbb{F}_n$ is a finite field, a well-known theorem by Wedderburn states that any division ring over $k$ is commutative. Since the quaternion algebras are noncommutative, they must all be isomorphic to the matrix ring, and thus $C_q\cong \M_2(k)$ for all values of $q$.
		\item If $k$ is an algebraic number field, i.e. a finite extension field of the rational numbers $\Q$, then there exist an infinite number of nonisomorphic quaternion algebras over $k$. Though there is no easy way to list the isomorphism classes, given any concrete couple of algebras of type $C_q$, it is possible to tell wether they are or not isomorphic in a finite number of steps by studying the existence of rational points in the conic given by equation \ref{eq:conic}. The existence (or not) of such solutions is obtained as a consequence of the Hasse-Minkowski principle and Hensel's lemma.
		\item If $k=\Q_p$ is a field of $p$--adic numbers, there is only one isomorphism class of quaternion algebras, which is never isomorphic to the matrix ring. This follows from the theory of quadratic forms over the $p$--adic numbers, see \cite[Chapter VI]{Lam05a} for details.
	\end{enumerate}
	
The remaining case, of field extensions of characteristic 2, cannot be described in such a fancy way using quaternion algebras. However, doing some computations (left to the reader) similar to the ones at the end of section \ref{ssec:duals_l}, we obtain the following result:

\begin{theorem}
	Let $k$ be a field with $\Char k=2$, and let $l$ and $l'$ be quadratic field extensions of $k$ generated by polynomials $p(x)=x^2+\alpha x + \beta$ and $p'(x)=x^2+\alpha' x + \beta'$. Then, all twisted tensor products $l\otimes_{\tau} l'$ of $l$ and $l'$ are described as
	\[
		D_q = k\langle x,y|\ x^2=\alpha x + \beta,\; y^2=\alpha' y + \beta',\; xy+yx=q \rangle,
	\]
	where the algebra $D_0$ corresponds to the usual tensor product $l\otimes l'$.
\end{theorem}

Let us finish the paper writing down the ``indecomposable'' (i.e., non-decomposable as a non-trivial twisted tensor product) algebras of dimension four over an algebraically closed field. A complete list of all algebras of dimension four can be found in \cite{gabriel74a}. We reproduce the scheme given there, highlighting the ``decomposable'' algebras putting them into a box.
$$\xymatrix@C=7pt@R=10pt{  &   &   &  &  &   &   & \cuadrado{k^4} \ar[d] &  \\
 &   &   &  &  &   &   & k^2\times k[\xi] \ar[rd]\ar[ld]&  \\
  \cuadrado{\hspace{0.1cm} \mathcal{M}_2(k) \hspace{0.1cm}}\ar[d]&  &   & \cuadrado{k\Gamma_{< 2}}\ar[rd] \ar[ld]  & &   &   k\times \displaystyle\frac{k[x]}{(x^3)}\ar[rd]\ar[dd] & & \cuadrado{k[\xi]\times k[\xi]} \ar[ld] \\
\cuadrado{kQ_{< 2}} \ar[rdddd]&   &  k(\Delta_1)_{< 2} \ar[rd] &  &  \cuadrado{k(\Delta_2)_{< 2}} \ar[dl] &   &   & \displaystyle\frac{k[x]}{(x^4)} \ar[d] &  \\
 &  G  \ar[ddrrrrrr] \ar[ddd] &   & \cuadrado{A_{0}} \ar[ddrrrr] &  &  & k\times\displaystyle\frac{k[x,y]}{(x,y)^2}\ar[rdd]  & \cuadrado{A_{1}} \ar[dd] & \cuadrado{A_{q}} \ar[ddl] \\
  &   &   &  &  &   &  &  &  \\
  &   &   &  &  &   &   & \displaystyle\frac{k[x,y]}{(x^3,xy,y^2)} \ar[dd] &  \\
   & \cuadrado{\wedge k^2} \ar[drrrrrr]  &   &  &  &   &   &  &  k\Sigma_{< 2} \ar[ld] \\
    &   &   &  &  &   &   & \displaystyle\frac{k[x,y,z]}{(x,y,z)^2} &  }$$
where:
 \begin{enumerate}
 \item The scalar $q\neq 1,-1,  0$.
 \item $G=\displaystyle\frac{k\langle x,y \rangle}{(x^2,y^2+xy,xy+yx)}$
 \item The quivers appearing above are the following:

  \[
    \begin{array}{clclcl}
        \Delta_1: &
            \xymatrix{ \circ \bucle \ar[r] & \circ} \qquad
        & \Delta_2: &
            \xymatrix{ \circ \bucle & \circ \ar[l] } \qquad
        & Q: &
            \xymatrix{\circ \ar@/^4pt/@<0.5ex>[r]& \circ \ar@/^4pt/@<0.5 ex>[l] } \vspace{0.5cm} \\
       \Sigma: &
            \xymatrix{\circ \ar@<0.5ex>[r] \ar@<-0.5ex>[r] & \circ}
        & \Gamma: &
           \xymatrix@R=10pt{\circ & \\ \circ  \ar[r] & \circ}
        &  &

       \end{array}
    \]

 \item The arrow $A\rightarrow B$ means that the algebra $B$ can be obtained by a degeneration of the structure of $A$.
\end{enumerate}
	


\begin{thebibliography}{10}

\bibitem{Agore07a}
 A. L. Agore, A. Chirvasitu, B. Ion, and G. Militaru,
 \newblock Factorization problems for finite groups,
 \newblock 	arXiv:math/0703471v2 [math.GR].

\bibitem{Simson}
 I. Assem, D. Simson and A. Skowro\'nski.
 \newblock Elements of the representation theory of associative algebras. Vol. 1:
 Techniques of representation theory.
 \newblock London Mathematical Society Student Texts, 65. Cambridge University Press, Cambridge, 2006.

\bibitem{Borowiec00a}
 A. Borowiec and W. Marcinek,
 \newblock On crossed product of algebras,
 \newblock {\em J. Math. Phys.} 41 (2000), 6959--6975.


\bibitem{Caenepeel00a}
 S. Caenepeel, B. Ion, G. Militaru, and S. Zhu,
 \newblock The factorisation problem and smash biproducts of algebras and coalgebras,
 \newblock {\em Algebr. Represent. Theory} 3 (2000), 19--42.

\bibitem{Caenepeel02a}
 S. Caenepeel, G. Militaru, and S. Zhu.
 \newblock Frobenius and separable functors for generalized module categories and nonlinear equations.
 \newblock Number 1787 in LNM. Springer-Verlag, 2002.

\bibitem{Cap95a}
 A. \u{C}ap, H. Schichl, and J. Van\u{z}ura,
 \newblock On twisted tensor products of algebras,
 \newblock {\em Comm. Algebra} 23 (1995), 4701--4735.

\bibitem{cibils}
C.~Cibils,
\newblock Non-commutive duplicates of finite sets,
\newblock {\em J. Algebra Appl} 5 (2006), 361--377.

\bibitem{gabriel74a} P. Gabriel,
\newblock Finite representation type is open,
\newblock {\em Proceedings of the International Conference on Representations of Algebras (Ottawa, 1974)}, Paper No. 10, 23 pp. Carleton Math. Lecture Notes, No. 9, Carleton Univ., Ottawa, Ont., 1974.


\bibitem{Gerstenhaber64a} M. Gerstenhaber,
 \newblock On the deformation of rings and algebras,
 \newblock {\em Ann. of Math.} 79 (1964), 59--103.


\bibitem{GTN} J. G\'{o}mez-Torrecillas and G. Navarro,
\newblock Serial coalgebras and their valued Gabriel quivers,
\newblock {\em J. Algebra} 319 (2008), 5039-5059.


\bibitem{Guccione99a}
 J. A. Guccione and J. J. Guccione,
 \newblock Hochschild homology of twisted tensor products,
 \newblock {\em K-Theory} 18 (1999), 363--400.


\bibitem{jlns} P. Jara, J. L\'{o}pez Pe\~{n}a, G. Navarro and D. \c{S}tefan,
\newblock On the classification of twisting maps between $K^n$ and $K^m$,
\newblock 	arXiv:0805.2874v1 [math.RA]

\bibitem{Lam05a}
T.~Y.~Lam.
\newblock {\em Introduction to Quadratic Forms over Fields}.
\newblock {Graduate Studies in Mathematics 67, American Mathematical Society 2005}.


\bibitem{Lewis06a}
D.~W.~Lewis.
\newblock{Quaternion Algebras and the Algebraic Legacy of Hamilton Quaternions},
\newblock {\em Irish Math. Soc. Bulletin} 57 (2006), 41--64.


\bibitem{Lopez08a}
 J. L\'opez Pe\~na and G. Navarro.
 \newblock On the classification and properties of noncommutative duplicates.
 \newblock {\em K-Theory} 38 (2008), 223--234.

\bibitem{Majid90b}
 S. Majid,
 \newblock Physics for algebraists: Non-commutative and non-cocommutative Hopf algebras by a bicrossproduct construction,
 \newblock {\em J. Algebra} 130 (1990), 17--64.

\bibitem{majid95a}
S. Majid,
\newblock Foundations of quantum group theory,
\newblock Cambridge University Press, Cambridge, 1995.

\bibitem{Pierce82a}
R.~S.~Pierce,
\newblock {\em Associative algebras},
\newblock {Graduate Texts in Mathematics 88, Springer-Verlag, 1982}.

\bibitem{Takeuchi81a}
 M. Takeuchi,
 \newblock Matched pairs of groups and bismash products of Hopf algebras,
 \newblock {\em Comm. Algebra} 9 (1981), 841--882.

\bibitem{tambara90a}
D. Tambara.
\newblock The coendomorphism bialgebra of an algebra.
\newblock {\em J. Fac. Sci. Univ. Tokyo Sect. IA Math.}  37  (1990), 425--456.

\bibitem{VanDaele94a} A. Van Daele and S. Van Keer,
 \newblock The Yang-Baxter and pentagon equation,
 \newblock {\em Compositio Math.}  91  (1994), 201--221.

\end{thebibliography}
\end{document}